\documentclass{amsart}
\usepackage{geometry}
\usepackage{enumerate}
\usepackage{amsmath}
\usepackage{amssymb}
\usepackage{amsthm}

\theoremstyle{plain}     
                         \newtheorem{theorem}             {Theorem}    [section]
\theoremstyle{definition}

\theoremstyle{plain}     
                         \newtheorem{lemma}      [theorem]{Lemma}
                         \newtheorem{proposition}[theorem]{Proposition}
\theoremstyle{remark}    
                         
                         \newtheorem{remark}     [theorem]{Remark}

\usepackage{ifthen}

\newcommand{\seclabel}   [1]{\label{sec:#1}}

\newcommand{\Secref}    [1]{Section~\ref{sec:#1}}

\newcommand{\eqnlabel}    [1]{\label{eqn:#1}}
\newcommand{\skipequation}[1]{\stepcounter{equation}\expandafter\global\expandafter\edef\csname r@eqn:#1\endcsname{\theequation}}

\newcommand{\eqnref}{\@ifstar\eqnrefnbr\eqnrefbr}
\newcommand{\eqnrefbr} [1]{(\ref{eqn:#1})}
\newcommand{\eqnrefnbr}[1]{\ref{eqn:#1}}

\newcommand{\thmlabel} [1]{\label{thm:#1}}

\newcommand{\lemlabel} [1]{\label{lem:#1}}

\newcommand{\Thmref}  [1]{Theorem~\ref{thm:#1}}

\newcommand{\Lemref}  [1]{Lemma~\ref{lem:#1}}

\let\tmpepsilon\epsilon
\let\epsilon\varepsilon
\let\varepsilon\tmpepsilon

\newcommand{\bE}{\mathbb E}
\newcommand{\bF}{\mathbb F}
\newcommand{\bG}{\mathbb G}
\newcommand{\bH}{\mathbb H}
\newcommand{\bI}{\mathbb I}

\newcommand{\bR}{\mathbb R}

\newcommand{\cB}{{\mathcal B}}

\newcommand{\cE}{{\mathcal E}}

\newcommand{\cH}{{\mathcal H}}

\newcommand{\cL}{{\mathcal L}}

\newcommand{\cN}{{\mathcal N}}

\newcommand{\cR}{{\mathcal R}}
\newcommand{\cS}{{\mathcal S}}

\def\int{\intop\limits}

\newcommand{\narrowarray}{\setlength{\arraycolsep}{0.2em}}

\narrowarray

\usepackage{dsfont}
\usepackage{mathrsfs}
\newcommand{\ba}{\begin{array}}
\newcommand{\ea}{\end{array}}

\newcommand{\E}{\mathbb{E}}
\newcommand{\F}{\mathbb{F}}
\newcommand{\G}{\mathbb{G}}
\newcommand{\I}{\mathbb{I}}
\newcommand{\N}{\mathbb{N}}

\newcommand{\R}{\mathbb{R}}

\newcommand{\U}{\mathbb{U}}

\newcommand{\eps}{\varepsilon}

\newcommand{\si}{\sigma}

\newcommand{\Ga}{\Gamma}
\newcommand{\Si}{\Sigma}
\newcommand{\Om}{\Omega}

\newcommand{\As}{A^{\scriptscriptstyle\Sigma}}
\newcommand{\ns}{\nu_{\scriptscriptstyle\Sigma}}
\newcommand{\cs}{c^{\scriptscriptstyle\Sigma}}

\newcommand{\csh}{\hat{c}^{\scriptscriptstyle\Sigma}}
\newcommand{\csb}{\bar{c}^{\scriptscriptstyle\Sigma}}
\newcommand{\ds}{d^{\scriptscriptstyle\Sigma}}
\newcommand{\fs}{f^{\scriptscriptstyle\Sigma}}

\newcommand{\gs}{g^{\scriptscriptstyle\Sigma}}

\newcommand{\us}{u^{\scriptscriptstyle\Sigma}}
\newcommand{\vs}{v^{\scriptscriptstyle\Sigma}}

\newcommand{\zs}{z^{\scriptscriptstyle\Sigma}}
\newcommand{\zsb}{\bar{z}^{\scriptscriptstyle\Sigma}}

\newcommand{\tgi}{\tilde{g}^{\scriptscriptstyle\Sigma}_-}

\newcommand{\hgi}{\hat{g}^{\scriptscriptstyle\Sigma}_-}

\newcommand{\hci}{\hat{c}^{\scriptscriptstyle\Sigma}_\text{in}}

\newcommand{\oc}{\bar{c}}
\renewcommand{\div}{\text{div\,}}
\newcommand{\pa}{\partial}
\title
	[Global Strong Solutions for Heterogeneous Catalysis Models]
	{Global Strong Solutions for a Class of \\ Heterogeneous Catalysis Models}

\author
	[Dieter Bothe]
	{Dieter Bothe}

\address
	{Center of Smart Interfaces and
	 Department of Mathematics \newline\indent
	 Technische Universit{\"a}t Darmstadt, \newline\indent
	 Alarich-Weiss-Str.~10, 64289 Darmstadt, Germany}
	
\email
	{bothe@csi.tu-darmstadt.de}

\author
	[Matthias K{\"o}hne]
	{Matthias K{\"o}hne}

\address
	{Mathematisches Institut \newline\indent
	 Heinrich-Heine-Universit{\"a}t D{\"u}sseldorf, \newline\indent
	 Universit{\"a}tsstr.~1, D-40225 D{\"u}sseldorf, Germany}
	
\email
	{koehne@math.uni-duesseldorf.de}

\author
	[Siegfried Maier]
	{Siegfried Maier}

\address
	{Mathematisches Institut \newline\indent
	 Heinrich-Heine-Universit{\"a}t D{\"u}sseldorf, \newline\indent
	 Universit{\"a}tsstr.~1, D-40225 D{\"u}sseldorf, Germany}
	
\email
	{maier@math.uni-duesseldorf.de}

\author
	[J{\"u}rgen Saal]
	{J{\"u}rgen Saal}

\address
	{Mathematisches Institut \newline\indent
	 Heinrich-Heine-Universit{\"a}t D{\"u}sseldorf, \newline\indent
	 Universit{\"a}tsstr.~1, D-40225 D{\"u}sseldorf, Germany}
	
\email
	{saal@math.uni-duesseldorf.de}



\date
	{\today}

\begin{document}
\setlength{\parskip}{0.5\baselineskip}
\setlength{\parindent}{0pt}
\renewcommand{\baselinestretch}{1.125}
\normalsize
\begin{abstract}
We consider a mathematical model for heterogeneous catalysis in a finite three-dimensional pore of cylinder-like geometry,
with the lateral walls acting as a catalytic surface.
The system under consideration consists of a diffusion-advection system inside the bulk phase and a
reaction-diffusion-sorption system modeling the processes on the catalytic wall and the exchange
between bulk and surface. We assume Fickian diffusion with constant coefficients, sorption kinetics 
with linear growth bound and a network of chemical reactions which possesses a certain triangular
structure. 
Our main result gives sufficient conditions for the existence of a unique global strong $L^2$-solution to this model,
thereby extending by now classical results on reaction-diffusion systems to the more complicated case of
heterogeneous catalysis.
\end{abstract}
\maketitle

\section*{Introduction}
Catalysis is a key technology in Chemical Engineering, employed not only to increase the speed of chemical
reactions by up to several orders of magnitude, but also to change the selectivity in favor of a desired product
against other possible output components of a chemical reaction network.
In heterogeneous catalysis, the catalytic substance forms a separate phase which is advantageous
concerning the separation of the products from the catalytic material. A prototypical setting,
which also underlies the mathematical model below, consists of a solid phase catalyst brought into contact with a gas or liquid which carries the educts as well as the product species inside the chemical reactor.
In this case, the overall chemical conversion consists of the following steps:
\begin{enumerate}
\item 
the educt species are transported to the surface of the catalytic substance;
\item 
molecules of at least one educt species adsorb at the catalyst surface;
\item 
adsorbed molecules react, either with other adsorbed molecules or
with molecules in the bulk phase directly adjacent to the surface;
\item 
the product molecules are desorbed.
\end{enumerate}
Of course, further processes will usually appear as well. For instance, adsorbed educt molecules may desorb
back into the bulk before a chemical reaction occurs, or they can be transported along the surface by means
of surface diffusion processes.
For a recent view on the complexity of heterogeneous catalysis modeling see \cite{Keil:Complexities}.

In the present paper, we only consider the case of pure surface chemistry, i.e.\
chemical reactions are only allowed between adsorbed species. This is actually no restriction, since one may
otherwise introduce an artificial adsorbed form of the reaction partner which is in the bulk adjacent to the surface and assign to it an infinite adsorption rate such that all arriving bulk molecules immediately adsorb
and, hence, are available for surface reaction.

In order for a heterogenous catalytic process to be efficient, a large surface area is required.
Therefore, in classical heterogeneous catalysis with solid phase catalyst, the latter is often provided as
a porous structure, e.g.\ in so-called packed-bed reactors. In this case, the smallest unit is a single
pore, into which the educts have to be transported in order to reach the pore wall, i.e.\ the catalytic surface.
More information on this classical reactor concept can be found, e.g., in \cite{Aris:Permeable-Catalysts},
\cite{Levenspiel:Chemical-Engineering} or \cite{White:HetCat}.
In recent years, with the advent of microreactor engineering technology, new reactor designs became feasible.
Due to the large area-to-volume ratio at the micro scale, multichannel microreactors with catalytic wall
coatings can replace classical porous structures and still provide fast and intense diffusive transport
to the channel walls in order to facilitate the reaction speed or selectivity enhancement; see, e.g.,
\cite{Wirth:Microreactors}, \cite{Renken:Catalytic-Microreactors}.
Since the given and precise structure of microreactors together with modern control and measurement techniques
allows for defined and reproducible operating conditions, this approach is much better accessible for
detailed quantitative modeling and simulation; cf.\ \cite{Bothe-Lojewski-Instantaneous},
\cite{Bothe-Lojewski-Parabolized}.
Structured catalytic microreactors are also employed for efficient screening of potential catalysts for
new reaction pathways; see, e.g., \cite{Jensen:HetCat}. To tap the full potential of such microsystems approaches and to intensify also more classical heterogeneous catalysis processes, realistic and sound mathematical models
are required as the basis for any numerical simulation. The most fundamental question then is whether a given
model is well-posed, a necessary requirement to enable any reasonable numerical treatment.

In what follows, we consider a single pore as a prototypical element, where we allow for convection
through the pore with a solenoidal velocity field which is assumed to be known and to satisfy the no-slip
boundary condition at fixed walls. We focus on pores having smoothly bounded cross shapes.
Let therefore $\Om:=A\times(-h,h)\subset \R^3$ denote a finite three-dimensional cylinder of height $2h>0$ with cross section $A\subset \R^2$ being a bounded simply connected $C^2$-domain, such that $\pa A$ is a closed regular $C^2$-curve. 
The boundary of $\Om$ decomposes into bottom $\Ga_\text{in}$, top $\Ga_\text{out}$ and lateral surface $\Si$, standing for inflow area, outflow area and active surface.
The mathematical model consists of the partial mass balances for all involved chemical components,
both within the bulk phase $\Om$ (representing the interior of the pore) and on the active surface $\Si$
(representing the catalytic surface). Inside the bulk phase, the species mass fluxes are due to
advection and diffusion, where we assume the latter to be governed by Fick's law.
On the active surface, we only consider diffusive fluxes along the surface, again assuming Fick's law
to be a reasonable constitutive relation. We allow for different diffusivities but the model ignores
cross-diffusion effects. Let us note in passing that for high surface coverage, cross-effects
between the transport of different constituents will appear which are not accounted for by our model.
The mass exchange between bulk and active surface is due to ad- and desorption phenomena, which are
usually modeled via kinetic relations in analogy to chemical reaction kinetics. Examples will be discussed below.

Insertion of the flux relations into the partial mass balances for continua yields the following mathematical
model for the unknown concentrations $(c_i,\cs_i)$ with $i=1,...,N$:
\begin{equation}\label{eq:cat}
\left\{\begin{array}{rclcl}
\pa_t c_i+(u\cdot \nabla) c_i- d_i\Delta c_i &=&0& \text{in}& (0,T)\times\Om, \\
\pa_t \cs_i - d^{\scriptscriptstyle\Si}_i \Delta_\Si \cs_i &=&r^\text{sorp}_i(c_i,\cs_i)+r^\text{ch}_i(\cs) & \text{on}& (0,T)\times \Si, \\
(u\cdot\nu)c_i-d_i\pa_\nu c_i&=&g^\text{in}_i& \text{on}& (0,T)\times \Ga_\text{in}, \\
-d_i\pa_\nu c_i &=&r^\text{sorp}_i(c_i,\cs_i)& \text{on}& (0,T)\times \Si,\\
-d_i\pa_\nu c_i &=& 0 & \text{on}& (0,T)\times \Ga_\text{out}, \\
-\ds_i\pa_{\nu_\Si} \cs_i &=&0 &\text{on}& (0,T)\times \pa \Si, \\
{c_i}|_{t=0}&=&c_{0,i}&\text{in}&\Om,\\
{\cs_i}|_{t=0}&=&\cs_{0,i} & \text{on}& \Si.
\end{array}\right.
\end{equation}
In \eqref{eq:cat}, the $d_i,\ds_i>0$ are given constant diffusivities, $u=u(t,x)$ denotes the velocity,
$r^\text{sorp}_i$ are the sorption rate functions, $r^\text{ch}_i$ the rates of molar mass production
due to surface chemistry, $g^\text{in}_i$ the inflow rates of molar mass, $c_{0,i}$ and $\cs_{0,i}$ the
initial concentrations in the bulk and on the active surface
and $\nu$ is the outer normal to $\Omega$.

{\bfseries ($\text{A}^\text{vel}$)}
Throughout this paper we assume that the velocity field satisfies
\begin{align}\label{eq:regularity-calss-velocity}
u \in \U^\Om_p(T):=W^{1,p}((0,T),L^p(\Om,\R^3))\cap L^p((0,T),W^{2,p}(\Om,\R^3)),
\end{align} for given time $T>0$. Moreover, we assume
\begin{align*} u\cdot \nu\leq 0 \quad \text{on}\; \Ga_\text{in}, \qquad u\cdot \nu =0\quad \text{on}\; \Si, \qquad u\cdot \nu\geq 0\quad \text{on}\;\Ga_\text{out},
\end{align*} and $\div u=0$ in the distributional sense.

\subsection*{Examples for sorption and reaction rates}
We give a few examples for sorption and reaction rate functions.

\begin{enumerate}
\item[(S1)] Let $k^\text{ad}_i,k^\text{de}_i>0$ denote adsorption and
desorption rate constants. The the simplest sorption rate is given by the linear {\it Henry law}, i.e.\
\begin{align*}
r^\text{sorp}_{H,i}(c_i,\cs_i)=k^\text{ad}_i c_i - k^\text{de}_i \cs_i.
\end{align*}
This law only applies for dilute systems.
\item[(S2)] For moderate concentrations, {\it Langmuir's law} given by
\begin{align*}
r^\text{sorp}_{L,i}(c_i,\cs_i)=k^\text{ad}_i c_i\left(1-\frac{\cs_i}{\cs_{\infty,i}}\right)-k^\text{de}_i\cs_i
\end{align*} may be employed. Here $\cs_{\infty,i}>0$ denotes the maximum capacity for adsorption of species $i$. 
In an application of our main results, we actually consider a modified version; see Remark \ref{rk:revisiting-examples}, which satisfies all of our assumptions on the sorption rate stated in Section \ref{sec:loc}.
\end{enumerate}

\begin{enumerate}
\item[(R1)] A standard example considers a reversible chemical reaction of type $A+B\rightleftharpoons P$
with $N=3$ components. If mass action kinetics is employed, the mass productions are governed by the rate function
\begin{align*}
r^\text{ch}(\cs)=\begin{pmatrix}
-k^\text{re}(\cs_1\cs_2-\kappa \cs_3)\\
-k^\text{re}(\cs_1\cs_2-\kappa \cs_3)\\
+k^\text{re}(\cs_1\cs_2-\kappa \cs_3) \end{pmatrix}.
\end{align*} Here $k^\text{re}>0$ denotes the rate constant of the
forward reaction, while $\kappa$ is the equilibrium constant for this reaction, determined as the ratio between forward and backward reaction rates.
\end{enumerate}

Due to the nonlinear coupling between bulk and surface in \eqref{eq:cat}, the extension of local and global
existence results from classical bulk reaction-diffusion systems (see \cite{Pie10} for a recent survey)
to the considered advection-diffusion-sorption-reaction system is not straightforward and there are only
few papers dealing with related models.
In \cite{Knabner-Otto}, a similar system but without chemical reactions has been studied. The authors have
shown that an $L^1$-contraction principle holds for the evolution operator and thereby weak solutions are unique. In \cite{Marpeau-Saad}, the case of fast sorption is considered in which local equilibrium between the
adsorbed concentration and the adjacent bulk concentration yields an algebraic relation between these quantities. Surface chemical reactions are not included.
As mentioned above, heterogeneous catalysis processes are often performed in porous media in which case
homogenization is a useful technique to obtain scale-reduced models. The mathematical details of such
a homogenization for periodic porous media have been worked out in \cite{Hornung}.
Another scale-reduced model of heterogeneous catalysis itself has been analyzed in \cite{Bothe:HetCat}
concerning the existence of time-periodic solutions.

\section{Main Results}\seclabel{results}

The main results of this paper are the local-in-time existence of a unique nonnegative strong $L^p$-solution
and the global-in-time existence of a unique nonnegative strong $L^2$-solution. 
In all sections, $(0,T)$ denotes a finite time interval. For a bounded simply connected $C^2$-domain $A\subset \R^2$, such that $\pa A$ is a closed regular $C^2$-curve, let $\Om:=A\times(-h,h)$ be a finite cylinder in $\R^3$.
The boundary of $\Om$ decomposes into three parts, which are
the lateral surface $\Si=\pa A\times(-h,h)$, an inflow area
$\Ga_\text{in}=A\times\{-h\}$ and an outflow area $\Ga_\text{out}=A\times\{h\}$, where
$\Ga_\text{in},\Ga_\text{out}$ are the bottom and the top of the cylinder $\Om$, respectively.
In particular, we have $\pa\Om=\Ga_\text{in}\cup\overline{\Si}\cup\Ga_\text{out}$.
The local existence result reads as follows. For the assumptions imposed on the sorption and reaction rates see the beginning of Section \ref{sec:loc}.

\begin{theorem}\thmlabel{Local-WP}(Local existence) Let $J=(0,T')\subset \R$ and $5/3<p<\infty$ with $p\neq 3$.
Suppose $u$ satisfies ($\text{A}^\text{vel}$), $r^\text{sorp}$ satisfies ($A_\text{F}^\text{sorp}$), ($A_\text{M}^\text{sorp}$), ($A_\text{B}^\text{sorp}$)
and $r^\text{ch}$ fulfills ($A^\text{ch}_\text{F}$), ($A^\text{ch}_\text{N}$), ($A^\text{ch}_\text{P}$).
Then for every set of data
\begin{align*}
g^\text{in}_i &\in W^{1/2-1/2p}_p(J,L^p(\Ga_\text{in}))\cap L^p(J,W^{1-1/p}_p(\Ga_\text{in})),\\
c_{0,i} &\in W^{2-2/p}_p(\Om),\\
\cs_{0,i} &\in W^{2-2/p}_p(\Si),
\end{align*}
which, if $p>3$, satisfies the compatibility conditions
\begin{align*}
\begin{array}{rcll}
c_{0,i} u(0) \cdot \nu -d_i\pa_\nu c_{0,i}&=&g^\text{in}_i (0)&\text{on}\;\Ga_{\textrm{in}},\\
-d_i\pa_\nu c_{0,i}&=&r^\text{sorp}_i(c_{0,i},\cs_{0,i})& \text{on}\;\Si,\\
-d_i\pa_\nu c_{0,i}&=&0&\text{on}\;\Gamma_{\textrm{out}},\\
-\ds_i\pa_{\nu_\Si} \cs_{0,i} &=&0 &\text{on}\; \pa \Si,
\end{array}
\end{align*}
there exists a $T^\ast\in(0,T')$ and a unique strong solution $(c_i,\cs_i)$ of (\ref{eq:cat})
satisfying
\begin{align*}
c_i&\in W^{1,p}((0,T),L^p(\Omega))\cap L^p((0,T),W^{2,p}(\Omega)),\\
\cs_i &\in W^{1,p}((0,T),L^p(\Si))\cap L^p((0,T),W^{2,p}(\Si)),
\end{align*} for all $T\in (0,T^*)$. If, in addition, $g^\text{in}_i\leq 0$ on $\Ga_\text{in}$, $c_{0,i}\geq
0$ in $\Om$ and $\cs_{0,i}\geq 0$ on $\Si$, then $c_i$ and $\cs_i $ are nonnegative.
\end{theorem}

\begin{remark} If $c_i,\cs_i$ are not continuous functions, then $c_i\geq 0$ has to be understood in the a.e.\ sense with respect to Lebesgue measure on $\Om$, and $\cs_i\geq 0$ with respect to the surface measure on $\Si$.
\end{remark}

The global existence is given by
\begin{theorem}\thmlabel{Global-WP} (Global existence)
	Let the assumptions of \Thmref{Local-WP} be satisfied for $p=2$
	and $T'=\infty$. Additionally, assume that $r^\text{ch}$ fulfills
	($\textrm{A}^{\textrm{ch}}_{\textrm{S}}$)
	(see (\ref{eq:A^ch_S-structure-of-reaction}))
	and that $g^\text{in}_i\in BC((0,T)\times\Ga_{\text{in}})$, $c_{0,i}\in BC(\Omega)$, $\cs_{0,i}\in
	BC(\Sigma)$,
	and that $-g^\text{in}_i$,
	$c_{0,i}$, and $\cs_{0,i}$ are nonnegative.
	Then the local solution $(c_i,\cs_i)$
	extends to a global solution of (\ref{eq:cat}), i.e., for $p=2$ the
	assertions of \Thmref{Local-WP} hold for every finite $T>0$.
\end{theorem}

\section{Notation}\seclabel{notation}
Let $X$ be a Banach space and let $A:D(A)\subset X\rightarrow X$ denote a linear and densely defined operator. Let $Y$ be another Banach space. We denote the space of bounded linear operators mapping from $X$ to $Y$ by $\mathscr{L}(X,Y)$.

All appearing constants, e.g.\ $C,M>0$ denote generic constants which may
vary from line to line, as long as it is not explicitly stated otherwise. 

When working in time-space sets we make use of the notation $\Om_T:=(0,T)\times\Om$ and $\Si_T:=(0,T)\times\Si$ for a finite $T>0$.
When dealing with half infinite cylinders we write,
$\Om_{(-\infty,h)}:=A\times(-\infty,h)\subset \R^3$ or $\Si_{(-\infty,h)}:=\pa A\times (-\infty,h)$ for its lateral surface.

For a Banach space $X$, a domain $G\subset \R^n$, $m\in\N_0:=\N\cup\{0\}$, and
$s\in(0,\infty)\setminus\N$, $W^{m,p}(G,X)$ denotes the usual $X$-valued Sobolev space and
$W^s_p(G,X)$ denotes the $X$-valued Sobolev-Slobodeckij space.
The norm of $X$ will be denoted by
$\Vert\cdot\Vert_X$. We also set $H^k:=W^{k,2}$.
In the same manner we employ $C^m$ for $m$-times continuously differentiable functions, and $BC^m$ for those with bounded derivatives up to order $m\in \N_0$. 

We denote by $f^+:=\max\{0,f\}$, $f^-:=-\min\{0,f\}$ the positive and negative part of a function $f$.
Moreover, we use the superscripts $\pm$ to denote sets of
functions, whose elements are nonnegative or nonpositive; e.g., we write
$ L^\infty(\Om)^+$ for functions which admit a bounded essential
supremum on $\Om$ and which are nonnegative a.e.\ in $\Om$. With
corresponding meaning we employ, e.g., $L^\infty(\Om)^-$. 

Throughout this work, let
$\nabla_\Si u:=(\nabla u)|_{\partial\Omega}
-\nu(\nu\cdot (\nabla u)|_{\partial\Omega})$
denote the surface gradient and let $\Delta_\Si u=\nabla_\Si \cdot \nabla_\Si u$ denote the Laplace-Beltrami operator on $\Si$.

\subsection{Maximal Regularity Spaces}
For $1<p<\infty$, we employ the following maximal regularity
spaces.
The solution spaces for the unknown functions $c_i,c^\Si_i$ are given by
\begin{align*}
\E^\Om_p(T)&:=W^{1,p}((0,T),L^p(\Omega))\cap L^p((0,T),W^{2,p}(\Omega)),\\
\E^\Si_p(T)&:= W^{1,p}((0,T),L^p(\Si))\cap L^p((0,T),W^{2,p}(\Si)).
\end{align*}
For the data spaces we first establish appropriate regularity classes.
Then we give necessary compatibility conditions in order to guarantee well-posedness. We set
\begin{align*}
\F^\Om_p(T)&:=L^p((0,T)\times \Om),\\
\F^\Si_p(T)&:=L^p((0,T)\times \Si),\\
\G^\text{in}_p(T)&:=W^{1/2-1/2p}_p((0,T),L^p(\Ga_\text{in})) \cap L^p((0,T),W^{1-1/p}_p(\Ga_\text{in})),\\
\G^\Si_p(T)&:= W^{1/2-1/2p}_p((0,T),L^p(\Si)) \cap L^p((0,T),W^{1-1/p}_p(\Si)),\\
\G^\text{out}_p(T)&:= W^{1/2-1/2p}_p((0,T),L^p(\Ga_\text{out})) \cap L^p((0,T),W^{1-1/p}_p(\Ga_\text{out})),\\
\I^\Om_p&:=W^{2-2/p}_p(\Om), \\
\I^\Si_p&:= W^{2-2/p}_p(\Si).
\end{align*} We define the tupel data space for the heterogeneous catalysis equations without initial data through
\begin{align*}
\F^{\Om,\Si}_p(T):=\F^\Om_p(T)\times\F^\Si_p(T)\times\G^\text{in}_p(T)\times\G^\Si_p(T)\times \G^\text{out}_p(T)\times\{0\}
\end{align*} and the tupel data space with initial spaces through
\begin{align*}
\F^{\Om,\Si}_{p,I}(T):= \F^{\Om,\Si}_p(T) \times\I^\Om_p\times\I^\Si_p.
\end{align*} In some statements we also employ the Dirichlet trace space on $\Si$, which is given by
\begin{align*}
\bH^\Si_p(T)&:= W^{1-1/2p}_p((0,T),L^p(\Si)) \cap L^p((0,T),W^{2-1/p}_p(\Si)).
\end{align*}

Moreover, we need subspaces of functions having zero time trace. For instance, the $X$-valued
Sobolev-Slobodeckij space with zero time trace is defined as
\begin{equation}\label{eq:zerotime} \,_0W^s_p((0,T),X):=\{c_i\in W^s_p((0,T),X):\; {c_i}|_{t=0}=0\}
\end{equation} for $s\in(0,\infty)\setminus\N$ and $sp>1$. We use this notation for all
appearing spaces whenever zero time trace makes sense and write, e.g., $\,_0\E^\Om_p(T)$, $_0\E^\Si_p(T)$ etc.



\section{Linear Equations}\label{sec:linear_equations}

We discuss a suitable linearization of the heterogeneous catalysis equations as it is given below in (\ref{eq:withPerturbation}) and show maximal regularity by
means of cylindrical $L^p$-theory, the surjectivity of the Neumann trace
operator, and a perturbation argument. For given data
\begin{equation*} (f_i,\fs_i,g^\text{in}_i,\gs_i,g^\text{out}_i,0,c_{0,i},\cs_{0,i})\in\F^{\Om,\Si}_{p,I}(T),
\end{equation*} we consider the linear system:
\begin{equation}\label{eq:withPerturbation}
\left\{\begin{array}{rclcl}
\pa_t c_i+ u\cdot\nabla c_i- d_i\Delta c_i &=&f_i& \text{in}& (0,T)\times\Om, \\
\pa_t \cs_i - d^{\scriptscriptstyle\Si}_i \Delta_\Si \cs_i &=&f^{\scriptscriptstyle\Si}_i & \text{on}& (0,T)\times \Si, \\
c_i u\cdot\nu-d_i\pa_\nu c_i&=&g^\text{in}_i& \text{on}& (0,T)\times \Ga_\text{in}, \\
-d_i\pa_\nu c_i &=&\gs_i& \text{on}& (0,T)\times \Si,\\
-d_i\pa_\nu c_i &=& g^\text{out}_i & \text{on}& (0,T)\times \Ga_\text{out}, \\
-d^{\scriptscriptstyle\Si}_i\pa_{\nu_\Si} \cs_i &=&0 &\text{on}& (0,T)\times \pa \Si, \\
c_i(0)&=&c_{0,i}   &\text{in}&\Om,   \\
\cs_i(0)&=&\cs_{0,i} & \text{in}& \Si.
\end{array}\right.
\end{equation} Our purpose is to solve (\ref{eq:withPerturbation}) for the unknown concentrations $c_i, \cs_i$.
The main result of this section is given by

\begin{proposition}{\label{prop:maxRegPerturbedSystem}} Let $5/3<p<\infty$ with  $p\neq 3$ and let $T>0$ be finite. Suppose the velocity field $u$ satisfies assumption ($A^\text{vel}$). Then (\ref{eq:withPerturbation}) admits a unique solution
\begin{equation*} (c_i,\cs_i)\in \E^\Om_p(T)\times \E^\Si_p(T)
\end{equation*} if and only if the data satisfy the regularity condition
\begin{equation*}
(f_i,\fs_i,g^\text{in}_i, \gs_i, g^\text{out}_i,0,c_{0,i},\cs_{0,i})\in\F^{\Om,\Si}_{p,I}(T),
\end{equation*} and in case of $p>3$ the compatibility conditions
\begin{align*}\begin{array}{rcll}
(u|_{t=0}\cdot\nu)c_{0,i}-d_i\pa_\nu c_{0,i}&=&g^\text{in}_i|_{t=0} &\text{on}\;\Ga_\text{in},\\
-d_i\pa_\nu c_{0,i}&=&\gs_i|_{t=0}& \text{on}\;\Si,\\
-d_i\pa_\nu c_{0,i}&=&g^\text{out}_i|_{t=0}& \text{on}\;\Ga_\text{out},\\
-\ds_i\pa_{\nu_\Si} \cs_{0,i}&=&0 &\text{on}\;\pa\Si.\end{array}
\end{align*} Additionally, the corresponding solution operator
${_0\cS}_T$ with respect to zero time trace spaces satisfies
\begin{equation*}
\Vert
{_0\cS}_\tau\Vert_{\mathscr{L}\left({{_0\F}^{\Om,\Si}_{p}(\tau)}^N,\,{_0\E^\Om_p(\tau)}^N\times{_0\E^\Si_p(\tau)}^N\right)}
\leq M
\end{equation*} with a constant $M>0$ independent of $\tau<T$.
\end{proposition}

Proposition~\ref{prop:maxRegPerturbedSystem} shows that the map $\cL : \E \to \F$,
where $\E= \E^\Om_p(T)^N\times \E^\Si_p(T)^N$ and $\F$ is the subspace of $\F^{\Om,\Si}_{p,I}(T)^N$
containing all elements which satisfy the required compatibility conditions,
is an isomorphism between the Banach spaces $\E$ and $\F$.
In this context, one speaks of {\em maximal regularity} of problem (\ref{eq:withPerturbation}),
or of the operator $\cL:\E\to\F$, acting on $(c_i,\cs_i)$.

{\it Plan of the proof:} System (\ref{eq:withPerturbation}) decomposes into two systems: One for the bulk concentrations $c_i$ in $\Om$ and one for the surface concentrations $\cs_i$ on $\Si$. In the first step we
neglect the velocity terms $(u\cdot \nabla ) c_i$ and $(u\cdot\nu)c_i$ -- playing the role of perturbation terms -- and consider only homogeneous boundary data, i.e.\ we start with
\begin{equation}\label{eq1.2}
\left\{\begin{array}{rclcl}
\pa_t c_i- d_i\Delta c_i &=&f_i& \text{in}& (0,T)\times\Om, \\
-d_i\pa_\nu c_i &=&0& \text{on}&  (0,T)\times\pa\Omega, \\
c_i(0)&=&c_{0,i}& \text{in}& \Om, \end{array} \right. \quad (i=1,...,N)
\end{equation} and
\begin{equation}\label{eq1.3}
\left\{\begin{array}{rclcl}
\pa_t \cs_i - d^{\scriptscriptstyle\Si}_i \Delta_\Si \cs_i &=&f^{\scriptscriptstyle\Si}_i & \text{on}& (0,T)\times \Si, \\
-d^{\scriptscriptstyle\Si}_i\pa_{\nu_\Si} \cs_i &=&0 &\text{on}& (0,T)\times \pa \Si, \\
\cs_i(0)&=&\cs_{0,i}&\text{on}& \Si. \end{array}\right. \quad (i=1,...,N)
\end{equation} We proceed as follows: We solve (\ref{eq1.2}) and (\ref{eq1.3}) separately via cylindrical $L^p$-theory. For more information on this topic, see \cite{NS2011}, \cite{NS2012}, \cite{Nau2012} and \cite{Nau2013}. Then a symmetric extension in axial direction of $\Om$ yields the surjectivity of the Neumann trace operator and,
consequently, the solvability of the inhomogeneous initial boundary value problem. By  perturbation arguments the obtained result carries over to (\ref{eq:withPerturbation}).

\subsection{Maximal Regularity of the Laplacian}\label{sec1}
We define
\begin{align*}
A_i:=-d_i \Delta, \quad A_i:D(A_i)\subset L^p(\Om)\rightarrow L^p(\Om), \\
D(A_i):=\{c_i\in W^{2,p}(\Om): \; -d_i\pa_\nu c_i=0\;\text{on}\; \pa\Om\},\quad(i=1,...,N)
\end{align*} and, analogously,
\begin{align*}
&\As_i:=-\ds_i\Delta_\Si, \quad \As_i:D(\As_i)\subset L^p(\Si)\rightarrow L^p(\Si), \\
&D(\As_i):=\{\cs_i\in W^{2,p}(\Si): \; -d^{\scriptscriptstyle\Si}_i\pa_{\nu_\Si}\cs_i=0\;\text{on}\;\pa\Si\}\quad (i=1,...,N).
\end{align*}
In this section we will show that $A_i$ admits
a bounded $\cH^\infty$-calculus and that $\As_i$ is $\cR$-sectorial with angles strictly less than $\pi/2$ which implies the
desired maximal regularity, cf.\ \cite{Denk-Hieber-Pruess:Maximal-Regularity}, \cite{Kunstmann-Weiss}, \cite{Weis-Maximal-Regularity-R-Sectoriality}. For $A_i$
we directly apply \cite[Theorem 4.1]{Nau2013}.
To this end, we employ the following cylindrical decomposition.
We set $V_1:=A\subset\R^2$ and $V_2:=(-h,h)\subset\R$ for the cylinder $\Om$ with height $2h>0$.
\begin{itemize}
\item
We first consider the problem which results on the sections for fixed $x_3$, yielding the following problems on
$V_1=A\subset \R^2$:
\begin{align*}
A_{i,1}:D(A_{i,1})\subset L^p(V_1)\rightarrow L^p(V_1), \quad A_{i,1}c_i:=-d_i(\pa_{x_1}^2+\pa_{x_2}^2)c_i,\\
D(A_{i,1}):=\{c_i\in W^{2,p}(V_1):\, B_{i,1} c_i:=-d_i\pa_\nu c_i
=0\; \text{on}\; \pa V_1\}.
\end{align*}
\item The direction along the cylinder leads to an operator on an interval $V_2$:
\begin{align*}
A_{i,2}:D(A_{i,2})\subset L^p(V_2)\rightarrow L^p(V_2), \quad A_{i,2}c_i:=-d_i \pa_{x_3}^2 c_i,\\
D(A_{i,2}):=\{ c_i\in W^{2,p}(V_2):\; B_{i,2} c_i
:=-d_i\pa_\nu c_i=0\;\text{on}\; \pa V_2\}.
\end{align*}
\end{itemize}
It is a well-known fact that both operators, $A_{i,1}$ and $A_{i,2}$
admit a bounded $\cH^\infty$-calculus with zero $\cH^\infty$-angle. For these results and an introduction of the $\cH^\infty$-calculus we refer to e.g.\ \cite{Denk-Hieber-Pruess:Maximal-Regularity} and \cite{Kunstmann-Weiss}.
Thus, we are exactly in the setting of \cite{Nau2013} in the case of the
strong Neumann-Laplacian given on both intersections  $V_1,V_2$.
Therefore \cite[Theorem 4.1 a)]{Nau2013} yields that $A_i+\delta$
for some $\delta>0$ admits a bounded $\cH^\infty$-calculus on
$L^p(\Om)$ with $\cH^\infty$-angle
$\phi^{\infty}_{A_i+\delta}<\frac{\pi}{2}$.
This implies maximal regularity for (\ref{eq1.2}) on finite intervals
$(0,T)$, see e.g.\ \cite{Denk-Hieber-Pruess:Maximal-Regularity}.

We continue with the discussion of $\As_i$.
Here we first employ a parametrization of the lateral surface and afterwards apply 
a result from \cite{Nau2012}.
By the assumption on the cross section of $\Om$, $\pa A$ constitutes a closed regular $C^2$-curve. Hence we may choose a parametrization $\psi^\Si:(0,2\pi)\rightarrow\R^2$ of $\pa A$. Let
\begin{equation*}\label{eq1.9}
\Psi^\Si:L^p(\pa V_1)\rightarrow L^p((0,2\pi)), \quad \Psi^\Si \cs_i:=\cs_i\circ \psi^\Si
\end{equation*} denote the induced pull-back. Then the $\Psi^\Si$-transformed Laplace-Beltrami operator $\Delta_{\Si,1}$ on $\pa V_1$ is a second-order elliptic differential operator on $L^p((0,2\pi))$ subject to periodic boundary conditions:
\begin{align*}
\As_{i,1}&:D(\As_{i,1})\subset L^p((0,2\pi))\rightarrow L^p((0,2\pi)), \quad
\As_{i,1} \cs_i:= \Psi^\Si(-\ds_i\Delta_{\Si,1}\cs_i),\\
D(\As_{i,1})&=\{\cs_i\in W^{2,p}((0,2\pi)): \;
\pa_\varphi^j \cs_i|_{\varphi=0}=
\pa_\varphi^j\cs_i|_{\varphi=2\pi}\;\;(j=0,1) \}.
\end{align*}
We have
\begin{lemma}\label{lm1.3} Let the pull-back $\Psi^\Si$ be given as above, then
\begin{enumerate}
\item[(1)] $\Psi^\Si$ is an isomorphism.
\item[(2)] $\Psi^\Si: D(\Delta_{\Si,1})\rightarrow D(\As_{i,1})$
is an isomorphism.
\end{enumerate}
\end{lemma}

Analogously to $A_i$, we resolve the transformation of $\As_i$ into two
cylindrical parts: $\As_{i,1}$ defined as above and $\As_{i,2}:=A_{i,2}$.

Applying this time \cite[Theorem 8.10]{Nau2012}
we infer that there is a shift $\delta>0$ such that
$\delta+ \As_i$ is $\cR$-sectorial with $\phi^\cR_{\As_i+\delta}<\frac{\pi}{2}$. This implies maximal regularity of (\ref{eq1.3}) on finite intervals $(0,T)$ by \cite[Theorem 4.2]{Weis-Maximal-Regularity-R-Sectoriality}.

We summarize the results of Subsection~\ref{sec1} in

\begin{lemma}\label{lm:Homogeneous-Equations} Let $1<p<\infty$ with $p\neq 3$ and $T>0$ be given.
\begin{enumerate}
\item[a)] System (\ref{eq1.2}) admits a unique solution $c_i\in \E^\Om_p(T)$ if and only if
the data satisfies the regularity condition $f_i\in\F^\Om_p(T)$, $c_{0,i}\in\I^\Om_p$ and, in case of $p>3$, the compatibility condition $-d_i\pa_\nu c_{0,i}=0$ on $\pa\Om$.
\item[b)] System (\ref{eq1.3}) admits a unique solution $\cs_i\in \E^\Si_p(T)$ if and only if
the data satisfies the regularity condition $\fs_i\in\F^\Si_p(T)$, $\cs_{0,i}\in\I^\Si_p$ and, in case of $p>3$, the compatibility condition $-\ds_i\pa_{\ns} c_{0,i}=0$ on $\pa\Si$.
\end{enumerate}
\end{lemma}


\subsection{Inhomogeneous Neumann Boundary Conditions}\label{sec4}

We turn to the discussion of inhomogeneous boundary values.
We show surjectivity of the trace operator
which leads to
the solvability of the corresponding inhomogeneous boundary value problem.

\begin{lemma}\label{lm:Surjectivity-of-Trace} Let $1<p<\infty$ with
$p\neq 3$ and let $T>0$. Then the Neumann trace operator
\begin{align*}
\gamma_1:\E^\Om_p(T)&\rightarrow \G^\text{in}_p(T)\times\G^\Si_p(T)\times\G^\text{out}_p(T)\\
&c\mapsto (\pa_\nu c|_{\Ga_\text{out}}, \pa_\nu c|_\Si,\pa_\nu c|_{\Ga_\text{in}})
\end{align*} is a retraction.
\end{lemma}

\begin{proof} Let $g^\text{in}\in \G^\text{in}_p(T),\; \gs\in\G^\Si_p(T)$ and $g^\text{out}\in\G^\text{out}_p(T)$.
We show that there exists a $c\in\E^\Om_p(T)$ such that
\begin{align}
\pa_\nu c&=g^\text{in}\quad \text{on}\;\Ga_\text{in},\label{eq:neum1}\\
\pa_\nu c&=\gs\quad\; \text{on}\;\Si,\label{eq:neum2}\\
\pa_\nu c&=g^\text{out}\;\;\; \text{on}\;\Ga_\text{out}. \label{eq:neum3}
\end{align} Note that we skip the index $i$ for better readability.
As before we work with $\Om=A\times(-h,h)$, such that $\Ga_\text{in}=A\times\{-h\}$ and $\Ga_\text{out}=A\times\{h\}$ in the present section. Let us define the halfspaces
\begin{equation*}
H_{-h}:=\R^2\times(-h,\infty),\quad H_{h}:=\R^2\times(-\infty,h).
\end{equation*} We proceed in two steps. \\
{\bfseries Step 1.} Due to \cite[Chapter 5]{Adams-Fournier:Sobolev-Spaces} there is an extension of $g^\text{in}$ to
\begin{equation*}
\tilde{g}^\text{in} \in \G^{\pa H_{-h}}_p(T)
\end{equation*} and of $g^\text{out}$ to
\begin{equation*}
\tilde{g}^\text{out} \in \G^{\pa H_{h}}_p(T).
\end{equation*} We have
\begin{equation*}
\tilde{g}^\text{in}(0) \in W^{1-3/p}_p(\pa H_{-h}), \quad \tilde{g}^\text{out}(0) \in W^{1-3/p}_p(\pa H_{h}).
\end{equation*} In case that $p>3$ we may choose a $\tilde{c}^\text{in}_0\in W^{2-2/p}_p(H_{-h})$ and $\tilde{c}^\text{out}_0\in W^{2-2/p}_p(H_{h})$ such that
\begin{align*}
\pa_\nu \tilde{c}^\text{in}_0&=\tilde{g}^\text{in}(0)\quad \text{on}\;\pa H_{-h}, \\
\pa_\nu \tilde{c}^\text{out}_0&=\tilde{g}^\text{out}(0)\quad \text{on}\;\pa H_{h},
\end{align*}
since $\partial_\nu : W^{2-2/p}_p(H_{-h}) \to W^{1-3/p}_p(\partial H_{-h})$ is 
a retraction due to \cite[Theorem 2]{Marschall87}.
In case $p<3$ we let
$\tilde{c}^\text{in}_0\in W^{2-2/p}_p(H_{-h})$ and $\tilde{c}^\text{out}_0\in W^{2-2/p}_p(H_{h})$ be arbitrary. Due to \cite[Proposition 4.6]{Denk-Hieber-Pruess:Maximal-Regularity-Inhomogeneous} we can solve the parabolic problem
\begin{equation*}
\left\{\begin{array}{rllll}\pa_t v^\text{in}-\Delta v^\text{in}&=&0&\text{in}&(0,T)\times H_{-h} \\
\pa_\nu v^\text{in} &=&\tilde{g}^\text{in}&\text{on}&(0,T)\times\pa H_{-h} \\
v^\text{in}(0)&=&\tilde{c}^\text{in}_0&\text{in}& H_{-h} \end{array}\right.
\end{equation*} with a unique
\begin{equation*}
v^\text{in} \in \E^{H_{-h}}_p(T)
\end{equation*} and analogously we solve
\begin{equation*}
\left\{\begin{array}{rllll}\pa_t v^\text{out}-\Delta v^\text{out}&=&0&\text{in}&(0,T)\times H_{h} \\
\pa_\nu v^\text{out} &=&\tilde{g}^\text{out}&\text{on}&(0,T)\times\pa H_{h} \\
v^\text{out}(0)&=&\tilde{c}^\text{out}_0&\text{in}& H_{h} \end{array}\right.
\end{equation*} with a unique
\begin{equation*}
v^\text{out} \in \E^{H_h}_p(T).
\end{equation*}
Let the cut-off function $\zeta \in C^\infty(\R,[0,1])$ satisfy
\begin{equation}\label{eq:retraction-cut-off-zeta}
\zeta|_{(-\infty,-h/3)}=1, \quad \zeta|_{(h/3,\infty)}=0.
\end{equation} Then the convex combination
\begin{equation*}
v:=\zeta v^\text{in}|_{\Om} + (1-\zeta) v^\text{out}|_\Om \in \E^\Om_p(T)
\end{equation*} fulfills the boundary conditions on top and bottom of $\Om$ by construction. \\
{\bfseries Step 2.} It remains to show that there exists a $w\in\E^\Om_p(T)$ such that
\begin{align*}
\pa_\nu w&=0\qquad\qquad \text{on}\;\Ga_\text{in},\\
\pa_\nu w&=\gs-\pa_\nu v\;\; \text{on}\;\Si,\\
\pa_\nu w&=0\qquad\qquad \text{on}\;\Ga_\text{out}.
\end{align*}
This problem can be reduced to an equation on a bounded $C^2$-domain, which works as follows.
To this end, first define $\Om_{-h}$ as the domain resulting from extending $\Om$
in some way boundedly and smoothly (at least in the $C^2$-sense) on the top. 
We also set $\Si_{-h}:=\partial\Omega_{-h}\setminus\overline{\Gamma}_{\text{in}}$. 
In a similar manner we define $\Om_{+h}$ and $\Si_{+h}$ by extending
$\Omega$ suitably at the bottom. Then, let $G_\pm$ denote the domains resulting from reflecting $\Om_{\pm h}$ at $h_\pm$ and set $\Ga_\pm:=\pa G_\pm$.
For instance, if the cross-section $A$ of $\Omega$ is a circle, we connect
half of a ball to $\Omega$ at $\Gamma_{\text{out}}$ resp.\
$\Gamma_{\text{in}}$. Then $G_\pm$ has the form of a `pill'.
It is clear that this way we always can find a suitable extension such that
$G_\pm$ is of class  $C^2$.

Let $\zeta$ be the cut-off function with the properties given in (\ref{eq:retraction-cut-off-zeta}). We extend $\zeta(\gs-\pa_\nu v)$ to
\begin{align*}
\tgi &\in \G^{\Si_{-h}}_p(T)\\
&=W^{1/2-1/2p}_p((0,T),L^p(\Si_{-h}))\cap L^p((0,T),W^{1-1/p}_p(\Si_{-h}))
\end{align*}
by $0$ and $\tgi$ to
\begin{equation*}
\hgi \in \G^{\Gamma_-}_p(T)
\end{equation*} by even reflection at $-h$.
Note that the extension by even reflection conserves the regularity
$W^{1-1/p}_p$ in axial direction.
Next, note that $\hgi(0)\in W^{1-1/3p}_p(\Gamma_-)$, if $p>3$. Hence we can use \cite[Theorem 2]{Marschall87} to obtain $\hci \in W^{2-2/p}_p(G_-)$ with
\begin{equation*}
\pa_\nu \hci =\hgi(0)\quad  \text{on}\;\Ga_-,
\end{equation*} if $p>3$.
Here we can assume w.l.o.g.\ that $\hci$ is even with respect to $\Ga_\text{in}$.
We solve the problem
\begin{equation}\label{eq:reflection-w^in}
\left\{\begin{array}{rllll}\pa_t w^\text{in}-\Delta w^\text{in}&=&0&\text{in}&(0,T)\times G_- \\
\pa_\nu w^\text{in} &=&\hgi&\text{on}&(0,T)\times\Gamma_- \\
w^\text{in}(0)&=&\hci&\text{in}& G_- \end{array}\right.
\end{equation} by \cite[Theorem 2.1]{Denk-Hieber-Pruess:Maximal-Regularity-Inhomogeneous} to obtain
\begin{equation*}
w^\text{in} \in \E^{G_-}_p(T).
\end{equation*}
Defining
\begin{equation*} \tilde{w}^\text{in}:=w^\text{in}|_{\Om_{-h}}
\in  \E^{\Om_{-h}}_p(T),
\end{equation*} we have $\pa_\nu \tilde{w}^\text{in}=\tgi$ on $\Si_{-h}$. Since $\hgi$ and $\hci$ are even in axial direction we have $\pa_\nu \tilde{w}^\text{in}=0$ on $\Ga_\text{in}$.
Analogously we proceed with $\Ga_\text{out}$. Here we extend
$(1-\zeta)(\gs-\pa_\nu v)$ to obtain a 
\begin{equation*}
\tilde{w}^\text{out}\in \E^{\Om_{+h}}_p(T)
\end{equation*} with 
$\pa_\nu \tilde{w}^\text{out}=(1-\zeta)(\gs-\pa_\nu v)$ on $\Si$ and $\pa_\nu \tilde{w}^\text{out}=0$ on $\Ga_\text{out}$.
Let the cut-off functions \\
$\tilde{\zeta}_1,\tilde{\zeta}_2\in C^\infty(\R,[0,1])$ satisfy
\begin{equation*}
{\tilde{\zeta}_1}=
\left\{\begin{array}{lll}
	1&: &(-\infty,h/2)\\
	0&:&(2h/3,\infty)	
\end{array}\right., \qquad {\tilde{\zeta}_2}=
\left\{\begin{array}{lll}
	0&:&(-\infty,-2h/3)\\
	1&:&(-h/2,\infty)
\end{array}\right..
\end{equation*} Then the combination
\begin{equation*}
w:=\tilde{\zeta}_1 \tilde{w}^\text{in}|_{\Om} + \tilde{\zeta}_2 \tilde{w}^\text{out}|_{\Om} \in \E^\Om_p(T)
\end{equation*} satisfies by construction $\pa_\nu w=0$ on $(0,T)\times\Ga_\text{in}$ and $\pa_\nu w=0$ on $(0,T)\times\Ga_\text{out}$. The remaining inhomogeneous boundary condition on $\Si$ is satisfied as well, since
\begin{align*}
\pa_\nu w&=\tilde{\zeta}_1\pa_\nu \tilde{w}^\text{in} + \tilde{\zeta}_2\pa_\nu \tilde{w}^\text{out}\\
&=\tilde{\zeta}_1\zeta(\gs-\pa_\nu v) +\tilde{\zeta}_2(1-\zeta )(\gs-\pa_\nu v)\\
&=\zeta (\gs-\pa_\nu v) + (1-\zeta)(\gs-\pa_\nu v) =\gs-\pa_\nu v \qquad \text{on}\;(0,T)\times\Si.
\end{align*}
Putting together step (i) and (ii) we define $c:=v+w \in \E^\Om_p(T)$
and obtain that $c$ satisfies (\ref{eq:neum1})-(\ref{eq:neum3}).
Thus, we have proved that 
there exists a bounded linear right-inverse to $\gamma_1$ which yields
that the trace operator 
$\gamma_1 : \E^\Om_p(T) \rightarrow \G^\text{in}_p(T)\times\G^\Si_p(T)\times\G^\text{out}_p(T)$ in fact is a retraction.
\end{proof}

We turn to the fully inhomogeneous Neumann system
\begin{equation}\label{eq:withoutPerturbation}
\left\{\begin{array}{rclcl}
\pa_t c_i- d_i\Delta c_i &=&f_i& \text{in}& (0,T)\times\Om, \\
\pa_t \cs_i - d^{\scriptscriptstyle\Si}_i \Delta_\Si \cs_i &=&f^{\scriptscriptstyle\Si}_i & \text{on}& (0,T)\times \Si, \\
-d_i\pa_\nu c_i&=&g^\text{in}_i& \text{on}& (0,T)\times \Ga_\text{in}, \\
-d_i\pa_\nu c_i &=&\gs_i& \text{on}& (0,T)\times \Si,\\
-d_i\pa_\nu c_i &=& g^\text{out}_i & \text{on}& (0,T)\times \Ga_\text{out}, \\
-d^{\scriptscriptstyle\Si}_i\pa_{\nu_\Si} \cs_i &=&0 &\text{on}& (0,T)\times \pa \Si, \\
c_i(0)&=&c_{0,i}   &\text{in}&\Om,   \\
\cs_i(0)&=&\cs_{0,i} & \text{in}& \Si.
\end{array}\right. \quad (i=1,...,N)
\end{equation}

\begin{lemma}\label{lm:withoutPerturbation}
Let $1<p<\infty$ with $p\neq 3$ and let $T>0$ be given. Then (\ref{eq:withoutPerturbation}) admits a unique solution
\begin{equation*} (c_i,\cs_i)\in \E^\Om_p(T)\times \E^\Si_p(T)
\end{equation*} if and only if the data satisfies the regularity conditions
\begin{align*}
(f_i,\fs_i,g^\text{in}_i,\gs_i,g^\text{out}_i,0,c_{0,i},\cs_{0,i})\in \F^{\Om,\Si}_{p,I}(T)
\end{align*} and, in case of $p>3$, the compatibility conditions
\begin{align*}\begin{array}{rcll}
-d_i\pa_\nu c_{0,i}&=&g^\text{in}_i|_{t=0}& \text{on}\;\Ga_\text{in},\\
-d_i\pa_\nu c_{0,i}&=&\gs_i|_{t=0}& \text{on}\;\Si,\\
-d_i\pa_\nu c_{0,i}&=&g^\text{out}_i|_{t=0}&  \text{on}\;\Ga_\text{out},\\
-\ds_i\pa_{\nu_\Si} \cs_{0,i}&=&0&  \text{on}\;\pa\Si.\end{array}
\end{align*} Additionally, the corresponding solution operator ${_0\cS}_T$ with respect to homogeneous initial values satisfies
\begin{equation}\label{eq:solution-operator-estimate}
\Vert {_0\cS}_\tau\Vert_{\mathscr{L}({{_0\F}^{\Om,\Si}_{p}(\tau)}^N,{_0\E^\Om_p(\tau)}^N\times{_0\E^\Si_p(\tau)}^N)}
\leq M\qquad (0<\tau<T)
\end{equation} for a constant $M>0$ independent of $\tau$.
\end{lemma}

\begin{proof}
By Lemma \ref{lm:Surjectivity-of-Trace} for given $g^\text{in}_i\in\G^\text{in}_p(T)$, $\gs_i\in\G^\Si_p(T)$, $g^\text{out}_i\in\G^\text{out}_p(T)$ there exists a $c^1_i\in\E^\Om_p(T)$ with
\begin{align*}
-d_i\pa_\nu c_i^1&=g^\text{in}_i \quad \text{on}\;(0,T)\times\Ga_\text{in}, \\
-d_i\pa_\nu c_i^1&=\gs_i \quad \text{on}\;(0,T)\times\Si, \\
-d_i\pa_\nu c_i^1&=g^\text{out}_i \quad \text{on}\;(0,T)\times\Ga_\text{out}.
\end{align*}
Secondly, due to Lemma \ref{lm:Homogeneous-Equations} for $f_i\in \F^\Om_p(T)$, $c_{0,i}\in\I^\Om_p(T)$ we find a unique $c^2_i\in\E^\Om_p(T)$ such that
\begin{equation*}
\left\{\begin{array}{rclcl}
\pa_t c_i^2- d_i\Delta c_i^2 &=&f_i-(\pa_t-d_i\Delta)c_i^1& \text{in}& (0,T)\times\Om, \\
-d_i\pa_\nu c_i^2&=&0& \text{on}& (0,T)\times \Ga_\text{in}, \\
-d_i\pa_\nu c_i^2 &=&0& \text{on}&  (0,T)\times\Si, \\
-d_i\pa_\nu c_i^2 &=&0& \text{on}&  (0,T)\times\Ga_\text{out}, \\
c_i^2|_{t=0}&=&c_{0,i}-c_i^1|_{t=0}& \text{in}& \Om. \end{array} \right.
\end{equation*} By construction $c_i:=c_i^1+c_i^2\in\E^\Om_p(T)$
satisfies (\ref{eq:withoutPerturbation}).
By employing the extension operator in zero time trace spaces from
\cite[Proposition 6.1]{Pruess-Saal-Simonett:Stefan-Analytic-Classical}
and that its norm is independent of $\tau<T$ the estimate for  
the solution operator readily follows.
\end{proof}

\subsection{Advection terms}\label{subsec:perturbation}
Let $5/3<p<\infty$ with $p\neq 3$ and $T>0$ be given. Assume that $u$ satisfies $(A^\text{vel})$.
We prove Proposition \ref{prop:maxRegPerturbedSystem} by a perturbation argument.
To this end, we have to show that the results obtained in Lemma \ref{lm:withoutPerturbation} carry over when adding the two perturbation terms $(u\cdot \nabla)c_i$ and $(u\cdot\nu)c_i$.
Let $\U^\Om_p(T)$ be given as in
(\ref{eq:regularity-calss-velocity}) and set
\begin{align*}
\U^\text{in}_p(T):=W^{1-1/2p}_p((0,T),L^p(\Ga_\text{in},\R^3))\cap L^p((0,T),W^{2-1/p}_p(\Ga_\text{in},\R^3))
\end{align*} for the Dirichlet trace space.

{\it Proof of Proposition \ref{prop:maxRegPerturbedSystem}.} The proof is carried out in three steps.

{\bfseries Step 1.} We estimate both perturbation terms occuring in (\ref{eq:withPerturbation}). Let $\eps\in(0,1)$ be sufficiently small. Then the following algebra properties hold for all $5/3<p<\infty$:
\begin{align*}
\U^\Om_p(T)&\cdot W^{1/2-\eps/2}_p((0,T),L^p(\Om,\R^3))\cap L^p((0,T),W^{1-\eps}_p(\Om,\R^3)) \hookrightarrow
\F^\Om_p(T),\\
\U^\text{in}_p(T)&\cdot
W^{1-1/2p-\eps/2}_p((0,T),L^p(\Ga_\text{in}))\cap L^p((0,T),
W^{2-1/p-\eps}_p(\Ga_\text{in}))\hookrightarrow \G^\text{in}_p(T).
\end{align*}
The first embedding follows by a direct calculation and the second
by taking trace of the embedding
\begin{align*}
\U^\Om_p(T)&\cdot W^{1-\eps/2}_p((0,T),L^p(\Om,\R^3))\cap
L^p((0,T),W^{2-\eps}_p(\Om,\R^3)) \\
&\hookrightarrow
W^{1/2}_p((0,T),L^p(\Omega))\cap L^p((0,T),W^{1,p}(\Omega)),
\end{align*}
which follows by a straight forward calculation, too.
For $0<\tau< T$ we infer the following estimates:
\begin{align*}
\Vert (u\cdot \nabla)c_i\Vert_{_0\F^{\Om}_p(\tau)} &\leq C \Vert u\Vert_{\U^\Om_p(T)} \Vert \nabla c_i\Vert_{{_0W}^{1/2-\eps/2}_p((0,\tau),L^p(\Om,\R^3))\cap L^p((0,\tau),W^{1-\eps}_p(\Om,\R^3))} \\
&\leq C \Vert u\Vert_{\U^p_p(T)} \Vert c_i\Vert_{{_0W}^{1-\eps/2}_p((0,\tau),L^p(\Om))\cap L^p((0,\tau),W^{2-\eps}_p(\Om))} \\
&\leq C\tau^\eta \Vert c_i\Vert_{_0\E^\Om_p(\tau)} \qquad\qquad (c_i\in{_0\E}^\Om_p(\tau))
\end{align*} with a constant $C>0$ and an exponent $\eta>0$ both being independent of $\tau$. Analogously
\begin{align*}
\Vert (u\cdot \nu)c_i\Vert_{_0\G^\text{in}_p(\tau)} &\leq C \Vert u\Vert_{\U^\text{in}_p(T)} \Vert c_i\Vert_{{_0W}^{1-1/2p-\eps/2}_p((0,\tau),L^p(\Ga_\text{in}))\cap L^p((0,\tau),W^{2-1/p-\eps}_p(\Ga_\text{in}))} \\
&\leq C \Vert u\Vert_{\U^\Om_p(T)} \Vert c_i\Vert_{{_0W}^{1-\eps/2}_p((0,\tau),L^p(\Om))\cap L^p((0,\tau),W^{2-\eps}_p(\Om))} \\
&\leq C\tau^\eta \Vert c_i\Vert_{_0\E^\Om_p(\tau)} \qquad\qquad (c_i\in{_0\E}^\Om_p(\tau))
\end{align*} with a constant $C>0$ and an exponent $\eta>0$ both being independent of $\tau$. It follows that the linear operator
\begin{align*}
\cB:{_0\E^\Om_p(\tau)}^N\!\!\!\times{_0\E^\Si_p(\tau)}^N\rightarrow{_0\F^{\Om,\Si}_{p}(\tau)}^N\\
\cB(c,\cs)=\left((u\cdot \nabla) c_i,\,0,\,(u\cdot\nu)c_i,\,0,\,0,\,0\right)_{i=1,...,N}
\end{align*} may be estimated by
\begin{align}\label{eq:perturbation-operator-estimate}
&\Vert \cB(c,\cs)\Vert_{{_0\F^{\Om,\Si}_{p}(\tau)}^N}  \nonumber \\
&\leq C\tau^\eta \Vert(c,\cs)\Vert_{{_0\E^\Om_p(\tau)}^N\times{_0\E^\Si_p(\tau)}^N}, \;\;((c,\cs)\in
{_0\E^\Om_p(\tau)}^N\!\!\!\times{_0\E^\Si_p(\tau)}^N)
\end{align} with a constant $C>0$ and an exponent $\eta>0$ both independent of $\tau<T$.

{\bf Step 2.} We give the construction of the solution of (\ref{eq:withPerturbation}) as a sum
$c_i=\hat{c}_i+\bar{c}_i$, $\cs_i=\csh_i+\csb_i$.
Let $(\hat{c}_i,\csh_i)\in \E^\Om_p(\tau)\times\E^\Si_p(\tau)$
be the solution to (\ref{eq:withoutPerturbation}) with $g^\text{in}_i$
replaced by some $\hat g^\text{in}_i$ satisfying the compatibility
condition $-d_i\partial_\nu c_{0,i}=\hat g^\text{in}_i(0)$ in case $p>3$
and which exists according to Lemma~\ref{lm:withoutPerturbation}.
Next, we set
\begin{align*}
\bar{f}_i=-(u\cdot\nabla)\hat{c}_i,\qquad
\bar{g}^\text{in}_i = g^\text{in}_i-\hat{g}^\text{in}_i
-(u\cdot\nu)\hat{c}_i.
\end{align*}
Note that then for $p>3$ the compatibility condition
\begin{align*}
\bar{g}^\text{in}_i|_{t=0}=g^\text{in}_i|_{t=0}-\hat{g}^\text{in}_i|_{t=0}-(u|_{t=0}\cdot\nu)c_{0,i} =0
\end{align*} is satisfied by construction.
Thus, the task is reduced to prove that for $0<\tau\leq T$
there exists a unique
solution $(\bar{c}_i,\csb_i)\in\E^\Om_p(\tau)\times\E^\Si_p(\tau)$ of
\begin{equation}\label{eq:cBar}
\left\{\begin{array}{rclcl}
\pa_t \bar{c}_i+(u\cdot\nabla)\bar{c}_i- d_i\Delta \bar{c}_i &=&\bar{f}_i& \text{in}& (0,\tau)\times\Om, \\
\pa_t \csb_i- d^{\scriptscriptstyle\Si}_i \Delta_\Si \csb_i &=&0& \text{on}& (0,\tau)\times \Si, \\
(u\cdot\nu)\bar{c}_i-d_i\pa_\nu \bar{c}_i&=&\bar{g}^\text{in}_i& \text{on}& (0,\tau)\times \Ga_\text{in}, \\
-d_i\pa_\nu \bar{c}_i &=&0& \text{on}& (0,\tau)\times \Si,\\
-d_i\pa_\nu \bar{c}_i &=&0 & \text{on}& (0,\tau)\times \Ga_\text{out}, \\
-d^{\scriptscriptstyle\Si}_i\pa_{\nu_\Si} \csb_i &=&0 &\text{on}& (0,\tau)\times \pa \Si, \\
\bar{c}_i|_{t=0}&=&0   &\text{in}&\Om,   \\
\csb_i|_{t=0}&=&0 & \text{in}& \Si.
\end{array}\right.
\end{equation}
This will be done in the final step.

{\bf Step 3.} We show the unique solvability of (\ref{eq:cBar}) on
some interval $(0,\tau)$. The proof will show that $\tau$ is independent
of the data
$f_i,\fs_i,g^\text{in}_i,\gs_i,g^\text{out}_i,c_{0,i},\cs_{0,i}$.
Due to the linearity of the system solvability then carries over
to the whole time interval $(0,T)$.

We apply a Neumann series argument to (\ref{eq:cBar}). To this end, let us reformulate (\ref{eq:cBar}) by means of the operators ${_0\cL}_\tau$ induced by the left-hand side of (\ref{eq:cBar}) and $\cB$. Let
\begin{equation*}
\bar{F}_i=(\bar{f}_i,0,\bar{g}^\text{in}_i,0,0,0)\in {_0\F}^{\Om,\Si}_p(\tau)^N, \qquad (i=1,...,N)
\end{equation*} such that (\ref{eq:cBar}) is equivalent to
\begin{equation*}
{_0\cL_\tau}(\bar{c},\csb)+\cB(\bar{c},\csb)=\bar{F} \qquad ((\bar{c},\csb)\in
{_0\E^\Om_p(\tau)}^N\!\!\!\times{_0\E^\Si_p(\tau)}^N).
\end{equation*} Due to
\begin{equation*}
{_0\cL_\tau} +\cB =(I+\cB\,{_0\cS}_\tau)\,{_0\cL}_\tau
\end{equation*} with the solution operator $_0\cS_\tau={_0\cL}_\tau^{-1}$ from Lemma \ref{lm:withoutPerturbation}, the invertibility of
$(I+\cB\,{_0\cS_\tau})$ from ${_0\F}^{\Om,\Si}_p(\tau)$ to
${_0\E^\Om_p(\tau)}^N\!\!\!\times{_0\E^\Si_p(\tau)}^N$ and, in turn, of ${_0\cL_\tau} +\cB $ from ${_0\E^\Om_p(\tau)}^N\!\!\!\times{_0\E^\Si_p(\tau)}^N$ to ${_0\F}^{\Om,\Si}_p(\tau)$ readily follows from (\ref{eq:perturbation-operator-estimate}) if we choose $\tau$ so small that $C\tau^\eta\,M<1$ with $M$ from  (\ref{eq:solution-operator-estimate}). Note that this is possible since $M$ is independent of $\tau<T$.\hfill $\square$

\section{Local well-posedness}\label{sec:loc}

In this section we derive unique solvability of
\begin{equation}\label{eq:locExSystem}
\left\{\begin{array}{rclcl}
\pa_t c_i+(u\cdot \nabla)c_i- d_i\Delta c_i &=&f_i& \text{in}& (0,T)\times\Om, \\
\pa_t \cs_i - d^{\scriptscriptstyle\Si}_i \Delta_\Si \cs_i &=&r^\text{sorp}_i(c_i,\cs_i)+r^\text{ch}_i(\cs) & \text{on}& (0,T)\times \Si, \\
(u\cdot\nu)c_i-d_i\pa_\nu c_i&=&g^\text{in}_i& \text{on}& (0,T)\times \Ga_\text{in}, \\
-d_i\pa_\nu c_i &=&r^\text{sorp}_i(c_i,\cs_i)& \text{on}& (0,T)\times \Si,\\
-d_i\pa_\nu c_i &=& 0 & \text{on}& (0,T)\times \Ga_\text{out}, \\
-d^\Si_i\pa_{\nu_\Si} \cs_i &=&0 &\text{on}& (0,T)\times \pa \Si, \\
c_i|_{t=0}&=&c_{0,i}&\text{in}&\Om,\\
\cs_i|_{t=0}&=&\cs_{0,i} & \text{on}& \Si,
\end{array}\right.
\end{equation} in the strong $L^p$-sense, locally in time. Recall when the index $i$ is used we mean $i=1,...,N$, and have e.g.
$c=(c_i)_{i=1,...,N}$ and $\cs=(\cs_i)_{i=1,...,N}$.

Sorption and reaction terms are required to satisfy the following assumptions:

\vspace{0,1cm}
{\bfseries ($\text{A}_\text{F}^{\text{sorp}}$)} The sorption rate acts as a function
satisfying
\begin{align*}
&r^\text{sorp}_i=r^\text{sorp}_i(c_i,\cs_i), \quad\!\! r^\text{sorp}_i\in C^2(\R^2), \quad\!\!
\nabla r^\text{sorp}_i\in BC^1(\R^2,\R^2).
\end{align*}

\vspace{0,1cm}
{\bfseries ($\text{A}_\text{M}^{\text{sorp}}$)}
The sorption rate increases monotonically in $c_i$ and decreases monotonically in $\cs_i$.

\vspace{0,1cm}
{\bfseries ($\text{A}_\text{B}^{\text{sorp}}$)}
The sorption rate admits linear bounds
\begin{align*}
-k^\text{de}_i \cs_i \leq r^\text{sorp}_i(c_i,\cs_i) \leq k^\text{ad}_i c_i \qquad (c_i,\cs_i\geq 0)
\end{align*} for given adsorption and desorption constants $k^\text{ad}_i,k^\text{de}_i>0$.

\vspace{0,25cm}
{\bfseries ($\text{A}^{\text{ch}}_\text{F}$)}
We assume that the chemical reactions fulfill
\begin{align*}
& r^\text{ch}_i=r^\text{ch}_i(\cs), \quad r^\text{ch} \in C^1([0,\infty)^N, \R^N), \qquad i=1,..,N.
\end{align*}

\vspace{0,1cm}
{\bfseries ($\text{A}^{\text{ch}}_\text{N}$)} The reaction is supposed to be quasi-positive, i.e.
\begin{align*}
r^\text{ch}_i(y)\geq 0, \qquad (y\in [0,\infty)^N, \;y_i=0).
\end{align*}

\vspace{0,1cm}
{\bfseries ($\text{A}^{\text{ch}}_\text{P}$)} The reaction admits polynomial growth, i.e.\ there exist a constant $M>0$ and an exponent $\gamma\in[1,\infty)$ if $p\in[2,\infty)$ and $\gamma=\frac{1}{1-p/2}\in [1,6]$ if $p\in(5/3,2)$, such that
\begin{align*}
|r^\text{ch}(y)|\leq M \left(1+|y|^\gamma\right) \qquad (y\in [0,\infty)^N).
\end{align*} Additionally, suppose the Jacobian fulfills
\begin{align*}
|({r^\text{ch}})'(y)|\leq M \left(1+|y|^{\gamma-1}\right) \qquad (y\in [0,\infty)^N ).
\end{align*}

\vspace{0,25cm}
\begin{remark}\label{rk:polynomial-growth} Let us comment on the polynomial growth conditions of $r^\text{ch}$ and $({r^\text{ch}})'$.
\begin{itemize}
\item[a)] In ($\text{A}^{\text{ch}}_\text{P}$) we restrict to $\gamma\in[1,6]$ if $p\in(5/3, 2)$. This is due to the embedding
\begin{equation*}
{_0\E}^\Si_p(T)\hookrightarrow L^{p\gamma}(\Si_T)
\end{equation*} for $p>5/3$, cf.\ \cite{Amann-Anisotropic-Function-Spaces}, which we employ in the proof of the local existence result. In case $p\geq 2$ only an arbitrary polynomial growth is required.
\item[b)] The growth rate $\gamma$ of $r^\text{ch}$ in ($\text{A}^{\text{ch}}_\text{P}$) yields that $r^\text{ch}$ acts as a Nemytskij operator
\begin{align*}
r^\text{ch}:{L^{p\gamma}(\Si_T)}^N\rightarrow {L^p(\Si_T)}^N,
\end{align*} cf.\ \cite[Theorem 3.1]{AZ90}. Analogously the growth rate $\gamma-1$ of $(r^\text{ch})'$ in ($\text{A}^{\text{ch}}_\text{P}$) yields
\begin{align*}
({r^\text{ch}})':L^{p\gamma}(\Si_T)^N\rightarrow L^{p\gamma/(\gamma-1)}(\Si_T)^{N\times N}.
\end{align*} In particular, $({r^\text{ch}})'$ maps a ball $\bar{B}_\delta$ of radius $\delta>0$ in $L^{p\gamma}(\Si_T)^N$ into a ball of radius $k(\delta)>0$ in $L^{p\gamma/(\gamma-1)}(\Si_T)^{N\times N}$. Hence an  application of the mean value theorem to the function $r^\text{ch}$ and H\"older's inequality with
\begin{align*}
\frac{1}{p}=\frac{1}{p\gamma/(\gamma-1)}+\frac{1}{p\gamma}
\end{align*} yields
\begin{align*}
\Vert r^\text{ch}&(\cs)-r^\text{ch}(\zs)\Vert_{L^p(\Si_T)^N} \\
&\leq
\sup_{\vs \in \bar{B}_\delta}\Vert (r^\text{ch})'(\vs)\Vert_{L^{p\gamma/(\gamma-1)}(\Si_T)^{N\times N}}\Vert \cs-\zs \Vert_{L^{p\gamma}(\Si_T)^N}\\
&\leq k(\delta) \Vert \cs-\zs \Vert_{L^{p\gamma}(\Si_T)^N} \qquad (\cs,\zs\in \bar{B}_\delta),
\end{align*} i.e.\ $r^\text{ch}$ acts as a locally Lipschitz continuous Nemytskij operator, cf.\ \cite[Theorem 3.10]{AZ90}.
\end{itemize}
\end{remark}

\vspace{0,25cm}
For $y\in\R^N$ let us denote $y^+=(y_i^+)_{i=1,...,N}$ where as before $y^+_i=\max\{0,y_i\}$. Since we do not know a priori whether a corresponding solution $(c,\cs)$ is nonnegative, we extend $r^\text{ch}$ as
\begin{align*}
r^\text{ch}_{i,+}:\R^N\rightarrow \R, \qquad r^\text{ch}_{i,+}(y):= r^\text{ch}_i({y}^+) \qquad (y\in \R^N).
\end{align*} Then (\ref{eq:locExSystem}) remains meaningful even if $c$ or $\cs$ take negative values.

The outcome of this section is the proof of \Thmref{Local-WP}. It
is based on maximal regularity of the linear system
(\ref{eq:withPerturbation}) proved in Section~\ref{sec:linear_equations}
and the contraction mapping principle. We start by proving the nonnegativity
of $c_i$ and of $\cs_i$.

\subsection{Nonnegativity of Concentrations}

For nonnegative initial
concentrations we have the following result.

\begin{lemma}\label{lm:nonneg}\lemlabel{Nonneg} (Nonnegativity)
Let $1<p<\infty$. Let $T>0$ and let $c_{0,i}\in \I^{\Om,+}_p$, $\cs_{0,i}\in \I^{\Si,+}_p$ and $g^\text{in}_i \in \G^\text{in}_p(T)^-$ be given. Suppose $u$ satisfies ($\text{A}^\text{vel}$), $r^\text{sorp}$ satisfies ($A_\text{F}^\text{sorp}$), ($A_\text{M}^\text{sorp}$), ($A_\text{B}^\text{sorp}$)
and $r^\text{ch}$ fulfills ($A^\text{ch}_\text{F}$), ($A^\text{ch}_\text{N}$), ($A^\text{ch}_\text{P}$).
Moreover, suppose $(c_i,\cs_i)\in\E^\Om_p(T)\times\E^\Si_p(T)$ is a strong $L^p$-solution of (\ref{eq:locExSystem}).  Then
\begin{align*}
c_i\geq 0 \quad \text{a.e.}\;\text{in}\;\Om_T, \qquad \cs_i \geq 0 \quad \text{a.e.}\; \text{on}\; \Si_T
\end{align*} hold true.
\end{lemma}

\begin{proof}
Let $\phi_\eps \in C^\infty(\R)$ be a pointwise approximation of
\begin{equation*}
\phi(r)=\left\{\ba{cr} -r & :r\leq 0 \\ 0 & :r>0 \ea\right.,
\end{equation*} as $\eps \rightarrow 0+$, which satisfies $\phi_\eps\geq 0$, $\phi_\eps'\leq 0$ and $\phi_\eps''\geq0$, e.g.
\begin{align*}
\phi_\eps(r):=\left\{
	\begin{array}{rl}
	-re^{\eps/r}&:r\leq 0\\
	0&:r>0
	\end{array}\right..
\end{align*}
Then we have for $c_i^-=\max\{0,-c_i\}$ that
\begin{align}\label{eq:properties_of_phi}
\phi_\eps (c_i) \rightarrow c_i^-, \qquad c_i \phi_\eps'(c_i) \rightarrow c_i
\left\{\ba{cr} -1&: c_i< 0\\ 0 &: c_i\geq 0 \ea\right\} =c_i^-
\end{align} as $\eps\rightarrow 0+$. We show
\begin{equation*}
\lim_{\eps\rightarrow 0}\left[\int_\Om \phi_\eps(c_i)dx+\int_\Si \phi_\eps(\cs_i) d\si\right]\leq 0.
\end{equation*} Applying $\phi_\eps$ to $c_i$ we obtain by partial integration
\begin{align}
\frac{d}{dt}\int_\Om \phi_\eps (c_i) dx &= \int_\Om \phi_\eps'(c_i) \pa_t c_i dx = \int_\Om \phi_\eps'(c_i) \div(d_i\nabla c_i -uc_i) dx\nonumber\\
&=\int_{\pa\Om} \phi_\eps'(c_i)(d_i \pa_\nu c_i -(u\cdot\nu)c_i) d\si - \int_\Om \phi_\eps''(c_i) \nabla c_i \cdot (d_i \nabla c_i - uc_i) dx\nonumber\\
&=-\int_{\Ga_\text{in}} \phi_\eps'(c_i) g^\text{in}_i d\si - \int_\Si \phi_\eps'(c_i) r^\text{sorp}_i d\si - \int_{\Ga_\text{out}} \phi_\eps'(c_i)(u\cdot\nu) c_i d\si\nonumber\\
&- d_i \int_\Om \phi_\eps''(c_i) |\nabla c_i|^2 dx + \int_\Om \phi_\eps''(c_i) \nabla c_i u c_i dx \label{eq:nonneg_bulk}
\end{align}due to the boundary conditions. In the same way we have
\begin{align}
\frac{d}{dt}\int_\Si \phi_\eps(\cs_i) d\si &= \int_\Si \phi_\eps'(\cs_i) \pa_t\cs_i d\si =\int_\Si \phi_\eps'(\cs_i)(\ds_i\Delta_\Si \cs_i +r^\text{sorp}_i(c_i,\cs_i)+r^\text{ch}_{i,+}(\cs))d\si\nonumber\\
&=-\ds_i\int_\Si \phi_\eps''(\cs_i) |\nabla_\Si\cs_i|^2 d\si+\int_\Si \phi_\eps'(\cs_i)(r^\text{sorp}_i(c_i,\cs_i)+r^\text{ch}_{i,+}(\cs))d\si.\label{eq:nonneg_surface}
\end{align} Let us go through all the integrals appearing on the right-hand side of (\ref{eq:nonneg_bulk}) and (\ref{eq:nonneg_surface}). The first and the fourth integrals on the right-hand side of (\ref{eq:nonneg_bulk}) and the first integral on the right-hand side of (\ref{eq:nonneg_surface}) are negative or zero such that we may drop them. The remaining four integrals are treated as follows: We combine the sorption boundary integrals to
\begin{equation*} \int_\Si (\phi_\eps'(\cs_i)-\phi_\eps'(c_i))r^\text{sorp}_i(c_i,\cs_i) d\si
\end{equation*} and split up this integral into three integrals on
\begin{equation*}
\{\text{sign}(c_i)=\text{sign}(\cs_i)\}, \qquad \{c_i\leq 0\leq \cs_i\},\qquad \{\cs_i\leq 0\leq c_i\}.
\end{equation*}
When $c_i$ and $\cs_i$ have the same sign this integral tends to $0$ as $\eps\rightarrow 0+$.
On $\{c_i\leq 0\leq \cs_i\}$ we have
\begin{align*} \int_\Si \mathds{1}_{\{c_i\leq 0\leq \cs_i\}}(\phi_\eps'(\cs_i)-\phi_\eps'(c_i))r^\text{sorp}_i(c_i,\cs_i) d\si\\ \rightarrow \int_\Si r^\text{sorp}_i(c_i,\cs_i) d\si \leq \int_\Si r^\text{sorp}_i(0,\cs_i) d\si \leq 0
\end{align*} as $\eps\rightarrow 0+$ by monotonicity of $r^\text{sorp}_i$ and $r^\text{sorp}_i (0,\cs_i)\leq 0$. In the same way on $\{\cs_i\leq 0\leq c_i\}$ we obtain
\begin{align*} \int_\Si\mathds{1}_{\{\cs_i\leq 0\leq c_i\}} (\phi_\eps'(\cs_i)-\phi_\eps'(c_i))r^\text{sorp}_i(c_i,\cs_i) d\si\\ \rightarrow -\int_\Si r^\text{sorp}_i(c_i,\cs_i) d\si \leq -\int_\Si r^\text{sorp}_i(c_i,0) d\si \leq 0
\end{align*} as $\eps\rightarrow 0+$ by monotonicity of $r^\text{sorp}_i$ and $r^\text{sorp}_i(c_i,0)\geq0$.
Because of $u\cdot \nu \geq 0$ on $\Ga_\text{out}$ we see that
\begin{align*}
-\int_{\Ga_\text{out}} \phi_\eps'(c_i) (u\cdot \nu) c_i d\si \rightarrow -\int_{\Ga_\text{out}} c_i^- (u\cdot\nu) d\si \leq 0.
\end{align*} We treat the reaction boundary integral by the quasi-positivity of $r^\text{ch}$ as follows.
We show
\begin{equation*}
\int_\Si \phi_\eps'(\cs_i) r^\text{ch}_{i,+}(\cs) d\si\leq 0
\end{equation*}
through
\begin{align*}
\int_\Si \phi_\eps'(\cs_i) r^\text{ch}_{i,+}(\cs) d\si
&=\int_\Si \phi_\eps'(\cs_i)\mathds{1}_{\{\cs_i>0\}} r^\text{ch}_{i,+}(\cs) d\si \\
&+\int_\Si \phi_\eps'(\cs_i) \mathds{1}_{\{\cs_i=0\}}r^\text{ch}_{i,+}(\cs) d\si
+ \int_\Si \phi_\eps'(\cs_i)\mathds{1}_{\{\cs_i<0\}} r^\text{ch}_{i,+}(\cs) d\si.
\end{align*} The first integral vanishes by the properties of $\phi_\epsilon'$, the second one is less than or equal to zero
by quasi-positivity and $\phi_\eps'(0)\leq 0$ as $\eps\rightarrow 0+$. The third one is less than or equal to zero by definition of the extension of $r^\text{ch}$ to $\R^N$, i.e.\ because of $\cs_i<0$ implies ${\cs_i}^+=0$ and $r^\text{ch}_i({\cs}^+)\geq 0$ by quasi-positivity.
We turn to the remaining integral $\int\nolimits_\Om \phi_\eps''(c_i) \nabla c_i u c_i dx$. Here we make use of $\nabla(\phi_\eps'(c_i))=\phi_\eps''(c_i)\nabla c_i$ and in the same manner of $\nabla(\phi_\eps(c_i))=\phi_\eps'(c_i)\nabla c_i$ and integrate by parts twice, such that
\begin{align*}
\int_\Om \phi_\eps''(c_i)\nabla c_i \cdot u c_i dx &=\int_\Om \nabla (\phi_\eps'(c_i))\cdot uc_i dx \\
&= \int_{\pa\Om} \phi_\eps' (c_i) (u\cdot \nu) c_i d\si - \int_\Om \phi_\eps'(c_i) \underbrace{\div(u c_i)}_{=u\cdot \nabla c_i} dx\\
&= \int_{\pa\Om} \phi_\eps' (c_i) (u\cdot \nu) c_i d\si - \int_\Om\nabla(\phi_\eps (c_i)) u dx\\
&= \int_{\pa\Om} \phi_\eps'(c_i)(u\cdot\nu) c_i d\si - \int_{\pa\Om} \phi_\eps(c_i)(u\cdot\nu) d\si,
\end{align*} where in the second and in the last step we made use of $\div u=0$. Employing
(\ref{eq:properties_of_phi}) we see
\begin{align*}
\int_\Om \phi_\eps''(c_i) \nabla c_i \cdot u c_i dx= \int_{\pa\Om} (\phi_\eps'(c_i)c_i - \phi_\eps(c_i))(u\cdot\nu) d\si \rightarrow 0
\end{align*} as $\eps\rightarrow 0+$. Therefore summing up (\ref{eq:nonneg_bulk}) and (\ref{eq:nonneg_surface}), integration in time over $[0,t]$ and taking the limit $\eps\rightarrow 0+$ yields
\begin{align*}
\int_\Om c_i^-(t) dx +\int_\Si {\cs_i}^-(t)d\si=\int_\Om \phi(c_i(t)) dx +\int_\Si \phi(\cs_i(t))d\si\\
\leq \int_\Om \phi (c_{0,i}) dx +\int_\Si \phi(\cs_{0,i})d\si= \int_\Om (c_{0,i})^- dx +\int_\Si (\cs_{0,i})^-d\si=0
\end{align*} which in turn gives $c_i^-=0$ a.e.\ in $\Om_T$, ${\cs_i}^-=0$ a.e.\ on $\Si_T$ and therefore $c_i\geq 0$ a.e.\ in $\Om_T$, $\cs_i\geq 0$ a.e.\ on $\Si_T$.
Note that for $\eps \rightarrow 0+$ we make use of Lebesgue's theorem on dominated convergence.
\end{proof}
\subsection{Existence of Solutions}
Let $T'>0$ be given and $T\leq T'$. Assume a set of (fixed) data
\begin{equation*} (f_i,0,g^\text{in}_i,0,0,0,c_{0,i},\cs_{0,i})\in\F^{\Om,\Si}_{p,I}(T')
\end{equation*}
to be given. We denote by
\begin{equation*} \cL_{T,i}:\E^\Om_p(T)\times \E^\Si_p(T)\rightarrow \F^{\Om,\Si}_{p}(T)
\end{equation*}
the isomorphism induced by Proposition~\ref{prop:maxRegPerturbedSystem},
that is, $\cL_{T,i}$ is the full linear operator on the right-hand side
of (\ref{eq:withPerturbation}) (except for the time traces).
The full nonlinear problem (\ref{eq:cat}) then is reformulated as
\begin{align}
	\cL_{T,i}(c_i,\cs_i)&=(f_i,0,g^\text{in}_i,0,0,0)
	+\cN_{T,i}(c,\cs),\label{reformnp}\\
	c_i(0)&=c_{0,i},\quad \cs_i(0)=\cs_{i,0},
	\quad i=1,\ldots,N,\nonumber
\end{align}
where $\cN_{T,i}$ includes the nonlinear sorption and reaction terms, i.e.,
\begin{align*}
\cN_{T,i}(c,\cs):=\big(0,r^\text{sorp}_i(c_i,\cs_i)+r^\text{ch}_{i,+}(\cs),0,r^\text{sorp}_i(c_i,\cs_i),0,0\big).
\end{align*}
In order keep the constants resulting from the estimates below independent of
$T$, we employ a suitable zero time trace splitting as described in
the following.

First we take care of the compatibility condition
arising from the nonlinear boundary condition on $\Si$. Taking time
trace results in $r^\text{sorp}_i(c_{0,i},\cs_{0,i})\in
W^{1-3/p}_p(\Si)$, which will be extended to $\G^\Si_p(T)$ by setting
\begin{equation*}
r^*_i:=\cR_\Sigma
e^{t\Delta_{\Si_{(-\infty,\infty)}}}\cE_{\Si_{(-\infty,\infty)}}
r^\text{sorp}_i(c_{0,i},\cs_{0,i}).
\end{equation*}
Here $\cE_{\Si_{(-\infty,\infty)}}$ denotes the extension operator
from the lateral surface $\Si$ to the surface of the infinite cylinder $\Si_{(-\infty,\infty)}$
and $\cR_\Sigma$ the corresponding restriction operator (note that both act
as bounded operators on the function classes considered here,
cf.\ \cite{Adams-Fournier:Sobolev-Spaces}).
Since $e^{t\Delta_{\Si_{(-\infty,\infty)}}}$ has the same regularizing
properties as the Laplacian on the whole space $\R^n$, for which the desired
regularity is well known \cite{Pruess-Saal-Simonett:Stefan-Analytic-Classical},
we see that $r^*_i\in \G^\Si_p(T)$.

Now we define the reference solution
$(c^*_i,{\cs_i}^*)\in\E^\Om_p(T')\times\E^\Si_p(T')$,
existing according to Proposition~\ref{prop:maxRegPerturbedSystem}, via
\begin{equation}\label{eq:OpLAndreferenceSolution}
\cL_{T,i}(c_i^*,{\cs_i}^*)=(f_i,0,g^\text{in}_i,r^*_i,0,0), \quad c_i(0)=c_{0,i} \quad \text{in}\; \Om,   \quad \cs_i(0)=\cs_{0,i} \quad \text{on}\;\Si.
\end{equation}
Decomposing $(c_i,\cs_i)$ as
\begin{equation*} c_i=\oc_i+c^\ast_i, \quad \cs_i=\csb_i+{\cs_i}^\ast
\end{equation*}
and subtracting (\ref{eq:OpLAndreferenceSolution}) from (\ref{reformnp}),
we end up with the reduced zero time trace problem
\begin{equation*}
{_0\cL}_{T,i}(\bar{c},\csb)={_0\cN}_{T,i}(\bar{c},\csb) \qquad (i=1,...,N).
\end{equation*}
Here ${_0\cL}_{T,i}$ denotes
the restriction of $\cL_{T,i}$ to
$\,_0\E^\Om_p(T)\times{_0\E^\Si_p(T)}$ and
\[
	{_0\cN}_{T,i}(\bar{c},\csb)
	:=\cN_{T,i}(\oc_i+c^\ast_i,\csb_i+{\cs_i}^\ast)
	-(0,0,0,r_i^\ast,0,0).
\]
Next, we define $\Phi_{T}:=(\Phi_{T,i})_{i=1,...,N}$ through
\begin{align*}
&{_0\Phi}_{T}:{\,_0\E^\Om_p(T)}^N\!\!\times{_0\E^\Si_p(T)}^N\rightarrow{\,_0\E^\Om_p(T)}^N\!\!\times{_0\E^\Si_p(T)}^N,\\
&{_0\Phi}_{T,i}(\bar{c},\csb):={_0\mathcal{S}}_{T,i}\,{_0\mathcal{N}}_{T,i}(\bar{c},\csb), \qquad (i=1,...,N),
\end{align*} with the bounded linear inverse ${_0\mathcal{S}}_{T,i}$ of ${_0\cL}_{T,i}$ given in Proposition \ref{prop:maxRegPerturbedSystem}.

{\it Proof of \Thmref{Local-WP}.} We apply the contraction mapping principle to ${_0\Phi}_T$, i.e., we show that there exists a
$\delta>0$, such that the mapping ${_0\Phi_T}$ constitutes a contraction on the closed ball $\bar{B}_\delta(0)\subset {\,_0\E^\Om_p(T)}^N\!\!\times{_0\E^\Si_p(T)}^N$ and fulfills ${_0\Phi}_T:\bar{B}_\delta(0)\rightarrow \bar{B}_\delta(0)$.

{\it (i) Contraction property:} Let $(\bar{c},\csb), (\bar{z},\zsb)\in
\bar{B}_\delta(0)$. Then we have
\begin{align}
&\Vert{_0\Phi}_T(\bar{c},\csb)-{_0\Phi}_T(\bar{z},\zsb)
\Vert_{{\,_0\E^\Om_p(T)}^N\!\!\times{_0\E^\Si_p(T)}^N}\nonumber\\
&\leq C \Vert{_0\mathcal{N}}_T(\bar{c},\csb)- {_0\mathcal{N}}_T(\bar{z},\zsb) \Vert_{{_0\F}^{\Om,\Si}_p(T)^N}\nonumber\\
&=C\Vert r^\text{sorp}(\bar{c}+c^*,\csb+{\cs}^*)-r^\text{sorp}(\bar{z}+c^*,\zsb+{\cs}^*) \Vert _{{(\F^\Si_p(T)\cap {_0\G}^\Si_p(T))}^N} \nonumber\\
&\qquad\strut+C\Vert r^\text{ch}(({\csb+{\cs}^*})^+)- r^\text{ch}(({\zsb+{\cs}^*})^+)\Vert_{\F^\Si_p(T)^N} \label{eq:ContractionEstimate}
\end{align}
with
\begin{align*}
C:=\sup \{\Vert {_0\mathcal{S}}_T\Vert_{\mathscr{L}({\,_0\F^{\Om,\Si}_{p}(T)^N},\,{\,_0\E^\Om_p(T)}^N\!\!\times{_0\E^\Si_p(T)}^N)}:\; T\in (0,T']\}
\end{align*} independent of $T$, cf.\! Proposition \ref{prop:maxRegPerturbedSystem}.
>From Remark \ref{rk:polynomial-growth} we infer that
\begin{align*}
\Vert r^\text{ch}(({\csb+{\cs}^*})^+)- r^\text{ch}(({\zsb+{\cs}^*})^+)\Vert_{L^p(\Si_T)^N}\leq L \Vert \csb-\zsb \Vert_{L^{p\gamma}(\Si_T)^N}
\end{align*} for a constant $L>0$ depending on $\delta$ but not $T$
and $\gamma$. Note in passing that we also used that $h\mapsto h^+$ is globally
Lipschitz continuous from $L^{p\gamma}(\Si_T)$ to $L^{p\gamma}(\Si_T)$ with
Lipschitz constant $1$. By the fact that $p>5/3$ we can estimate as
\begin{align*}
\Vert \csb-\zsb\Vert_{L^{p\gamma}(\Si_T)^N} \leq K T^\eta \Vert \csb-\zsb \Vert_{{_0\E}^\Si_p(T)^N}
\end{align*} with a constant $K>0$ and an exponent $\eta>0$ independent of $T$. We arrive at
\begin{align}\label{eq:ReactionEstimate}
\Vert r^\text{ch}(({\csb+{\cs}^*})^+)- r^\text{ch}(({\zsb+{\cs}^*})^+)\Vert_{\F^\Si_p(T)^N}\leq LK T^\eta \Vert \csb-\zsb\Vert_{_0\E^\Si_p(T)^N}.
\end{align} We turn to the estimate of the sorption rate. By
($\text{A}_\text{F}^{\text{sorp}}$), ($\text{A}_\text{B}^{\text{sorp}}$)
the mapping $r^\text{sorp}_i$ acts as a Nemytskij operator from
$W^s_p(\Si_T)\times W^s_p(\Si_T)$ to $W^s_p(\Si_T)$ for $s\in(0,1)$,
cf.\ Section~3.1 in \cite{Sickel96}.
Note that ($\text{A}_\text{B}^{\text{sorp}}$) implies $r^\text{sorp}_i(0,0)=0$.
An application of the mean value theorem to the function $r^\text{sorp}_i$ and using $\nabla r^\text{sorp}_i\in BC^1(\R^2,\R^2)$ yields the global Lipschitz continuity of its induced Nemytskij operator. Hence we may employ
\begin{align*}
W^s_p(\Si_T)=W^s_p((0,T),L^p(\Si))\cap L^p((0,T),W^s_p(\Si))
\end{align*} for $s\in(0,1)$ and 
\begin{align*}
{_0\bH^\Si_p}(T)\times {_0\E^\Si_p}(T)\hookrightarrow \big({_0W^{1-1/p+\epsilon}_p}((0,T), L^p(\Si))\cap L^p((0,T),W^{1-1/p+\epsilon}_p(\Si))\big)^2 \\
 \overset{r^\text{sorp}_i}{\longrightarrow} {_0W^{1-1/p+\epsilon}_p}((0,T), L^p(\Si))\cap L^p((0,T), W^{1-1/p+\epsilon}_p(\Si))
 \hookrightarrow {_0\G^\Si_p}(T),
\end{align*}
for sufficiently small $\epsilon>0$, such that we obtain, similarly as for the estimate of the reaction term,
\begin{align}
\Vert r^\text{sorp}_i(\bar{c}+c^*,\csb+{\cs}^*)-r^\text{sorp}_i(\bar{z}+c^*,\zsb+{\cs}^*) \Vert _{ {_0\G}^\Si_p(T)^N} \nonumber\\
\leq L'K'T^\eta \left(\Vert \bar{c}_i-\bar{z}_i\Vert_{_0\E^\Om_p(T)} +\Vert \csb_i-\zsb_i \Vert_{_0\E^\Si_p(T)} \right) \label{eq:SorptionEstimate}
\end{align} with constants $L',K'>0$ and an exponent $\eta>0$
independent of $T<T'$. Combining (\ref{eq:ReactionEstimate}) and
(\ref{eq:SorptionEstimate}) yields
\begin{align*}
&\Vert{_0\Phi}_T(\bar{c},\csb)-{_0\Phi}_T(\bar{z},\zsb)
\Vert_{{\,_0\E^\Om_p(T)}^N\!\!\times{_0\E^\Si_p(T)}^N}\\
&\leq C(LK+L'K')T^\eta\,\Vert (\bar{c},\csb)-(\bar{z},\zsb)\Vert_{{\,_0\E^\Om_p(T)}^N\!\!\times{_0\E^\Si_p(T)}^N}
\end{align*}
for $(\bar{c},\csb), (\bar{z},\zsb)\in \bar{B}_\delta(0)$.
We choose $T$ so small that
\begin{align}\label{eq:ChoosingT}
C(LK+L'K')T^\eta\leq \frac{1}{2},
\end{align} which is possible since all other constants appearing in (\ref{eq:ChoosingT}) are independent of $T<T'$.
Hence ${_0\Phi}_T$ is a contraction on $\bar{B}_\delta(0)$.

\vspace{0,25cm}
{\it (ii) Self mapping property:} Let $(\bar{c},\csb)\in \bar{B}_\delta(0)$.
Then we have
\begin{align*}
&\Vert {_0\Phi}_T(\bar{c},\csb)\Vert_{{\,_0\E^\Om_p(T)}^N\!\!\times{_0\E^\Si_p(T)}^N} \leq C\Vert {_0\mathcal{N}}_T(\bar{c}+c^*,\csb+{\cs}^*)\Vert_{_0\F^{\Om,\Si}_p(T)^N} \\
&\leq C\Vert r^\text{sorp}(\bar{c}+c^*,\csb+{\cs}^*)-r^* \Vert_{_0\G^\Si_p(T)^N}
+C\Vert r^\text{ch}({(\csb+{\cs}^*)}^+) \Vert_{\F^\Si_p(T)^N}.
\end{align*}
Analogously to (i) we estimate the reaction term by
\begin{align}
&\Vert r^\text{ch}(({\csb+{\cs}^*})^+)\Vert_{L^p(\Si_T)^N}\nonumber\\
&\leq \Vert r^\text{ch}(({\csb+{\cs}^*})^+)-r^\text{ch}({({\cs}^*)}^+)\Vert_{L^p(\Si_T)^N}+\Vert r^\text{ch}(({\cs}^*)^+)\Vert_{L^p(\Si_T)^N}\nonumber\\
&\leq L \Vert \csb\Vert_{L^{p\gamma}(\Si_T)^N}+\Vert r^\text{ch}(({\cs}^*)^+)\Vert_{L^p(\Si_T)^N}\nonumber\\
&\leq L K T^\eta \Vert \csb\Vert_{{_0\E^\Si_p(T)}^N}+\Vert r^\text{ch}(({\cs}^*)^+)\Vert_{L^p(\Si_T)^N} \label{eq:SelfmappingReaction}
\end{align} with constants $L,K>0$ being independent of
$T$. In the same manner as in
(i) for the sorption term we obtain
\begin{align}
&\Vert r^\text{sorp}_i(\bar{c}_i+c_i^*,\csb_i+{\cs_i}^*)-r^*_i\Vert_{_0\G^\Si_p(T)}
\nonumber\\
&\leq \Vert
r^\text{sorp}_i(\bar{c}_i+c_i^*,\csb_i+{\cs_i}^*)-r^\text{sorp}_i(c_i^*,{\cs_i}^*)\Vert_{_0\G^\Si_p(T)}\nonumber\\
&\qquad\strut +\Vert r^\text{sorp}_i(c_i^*,{\cs_i}^*)\Vert_{\G^\Si_p(T)} +\Vert r_i^*\Vert_{_0\G^\Si_p(T)}\nonumber\\
& \leq L'K'T^\eta\Vert (\bar{c}_i,\csb_i)\Vert_{{\,_0\E^\Om_p(T)}^N\!\!\times{_0\E^\Si_p(T)}^N}
 +\Vert r^\text{sorp}_i(c_i^*,{\cs_i}^*)\Vert_{\G^\Si_p(T)} +\Vert r_i^*\Vert_{_0\G^\Si_p(T)} \label{eq:SelfmappingSorption}
\end{align}
with constants $L',K'>0$ and an exponent $\eta>0$ being all independent of $T$. Putting together (\ref{eq:SelfmappingReaction}) and (\ref{eq:SelfmappingSorption}) yields
\begin{align*}
\Vert
{_0\Phi}_T(\bar{c},\csb)\Vert_{{\,_0\E^\Om_p(T)}^N\!\!\times{_0\E^\Si_p(T)}^N}
&\leq C(LK+L'K')T^\eta \delta +\Vert r^\text{ch}(({\cs}^*)^+)\Vert_{L^p(\Si_T)^N}\\
&\qquad\strut+\Vert r^\text{sorp}_i(c_i^*,{\cs_i}^*)\Vert_{\G^\Si_p(T)} +\Vert r_i^*\Vert_{_0\G^\Si_p(T)}.
\end{align*}
Since $c_i^*$, ${\cs_i}^*$, and $r_i^*$ are fixed functions,
notice that the latter three terms can be made small by choosing
$T>0$ small. Thus, by choosing $T$ so small that the sum of those
three terms is less than $\delta/2$ and such that (\ref{eq:ChoosingT})
is satisfied, we arrive at
\begin{align*}
\Vert {_0\Phi}_T(\bar{c},\csb)\Vert_{{\,_0\E^\Om_p(T)}^N\!\!\times{_0\E^\Si_p(T)}^N} \leq\delta \qquad
((\bar{c},\csb) \in \bar{B}_\delta(0)).
\end{align*}
The proof of \Thmref{Local-WP} is now complete.
\hfill $\square$

\section{Global Well-Posedness}
\seclabel{Global-WP}
In this section we show that the local-in-time strong solutions obtained in \Thmref{Local-WP}
in fact exist globally, provided the reaction rates satisfy some structural condition that allows for
the derivation of a priori estimates.
We suppose in addition to ($\textrm{A}^{\textrm{ch}}_{\textrm{F}}$), ($\textrm{A}^{\textrm{ch}}_{\textrm{N}}$) and ($\textrm{A}^{\textrm{ch}}_{\textrm{F}}$)
the following assumption on the structure of the reaction rates.

		\vspace*{0.5em}
		{\bfseries ($\textrm{A}^{\textrm{ch}}_{\textrm{S}}$)}
		There exists a lower triangular matrix $Q=(q_{ij})_{1\leq i,j\leq N} \in \bR^{N \times N}$
		with strictly positive diagonal entries such that
		\begin{equation}\label{eq:A^ch_S-structure-of-reaction}
			Q r^{\textrm{ch}}(y) \leq C \left( 1 + \sum^N_{j = 1} y_j \right) v, \qquad y \in [0,\infty)^N
		\end{equation}
		for some constant $C > 0$ and $v = (1,\,\dots,\,1)$.
		\vspace*{0.5em}

\smallskip
Reaction-diffusion systems with this triangular condition have been
widely studied by several authors; see \cite{Pie10} and the references
cited therein. When proving global existence results, condition (\ref{eq:A^ch_S-structure-of-reaction}) allows for an iteration scheme which has been applied successfully in many situations. A major objective of this section is to generalize this iteration scheme for standard reaction-diffusion systems subject to ($\textrm{A}^{\textrm{ch}}_{\textrm{S}}$) to heterogeneous catalysis. The main difference compared to standard systems lies in the fact that the reaction takes place on the boundary instead inside the bulk and that we also have to deal with terms arising from sorption processes.

In this setting, we obtain the global-in-time well-posedness result, given by \Thmref{Global-WP}.
The proof of \Thmref{Global-WP} requires, besides the maximal regularity estimates obtained in \Secref{linear_equations},
also some comparison principles and some weak-type estimates, which are provided by the following results.
\begin{lemma}\lemlabel{Linear-Positivity}
	Let $T > 0$ and let $1 < p < \infty$.
	Let $\alpha,\,\beta > 0$. \vspace*{-0.5\baselineskip}
	\begin{enumerate}[a)]
		\item Assume $f \in \bF^\Omega_p(T)$, $g^{\textrm{in}} \in \bG^{\textrm{in}}_p(T)$, $\gs \in \bG^\Sigma_p(T)$
			and $v_0 \in \bI^\Omega_p$ with
			\begin{equation*}
				(u \cdot \nu) v_0 - \beta \partial_\nu v_0 = g^{\textrm{in}}(0) \quad \mbox{on}\ \Gamma_{\textrm{in}}, \qquad
				- \beta \partial_\nu v_0 = \gs(0) \quad \mbox{on}\ \Sigma,
			\end{equation*}
			if $p > 3$.
			Let $v \in \bE^\Omega_p(T)$ be a strong solution to
			\begin{equation*}
				\left\{\begin{array}{rclll}
					\alpha \partial_t v + (u \cdot \nabla) v - \beta \Delta v & = & f               & \quad \mbox{in} & (0,\,T) \times \Omega,                \\[0.5em]
				                     (u \cdot \nu) v - \beta \partial_\nu v & = & g^{\textrm{in}} & \quad \mbox{on} & (0,\,T) \times \Gamma_{\textrm{in}},  \\[0.5em]
				                                     - \beta \partial_\nu v & = & \gs        & \quad \mbox{on} & (0,\,T) \times \Sigma,                \\[0.5em]
				                                     - \beta \partial_\nu v & = & 0               & \quad \mbox{on} & (0,\,T) \times \Gamma_{\textrm{out}}, \\[0.5em]
					                                                     v(0) & = & v_0             & \quad \mbox{in} & \Omega.
				\end{array}\right.
			\end{equation*}
			If $f,\,v_0 \geq 0$ and $g^{\textrm{in}},\,\gs \leq 0$, then $v \geq 0$. \vspace*{0.5\baselineskip}
		\item Assume $f \in \bF^\Sigma_p(T)$ and $v_0 \in \bI^\Sigma_p$ with $- \beta \partial_{\nu_\Sigma} v_0 = 0$ on $\partial \Sigma$, if $p > 3$.
			Let $v \in \bE^\Sigma_p(T)$ be a strong solution to
			\begin{equation*}
				\left\{\begin{array}{rclll}
					\alpha \partial_t v - \beta \Delta_\Si v & = & f   & \quad \mbox{on} & (0,\,T) \times \Sigma,          \\[0.5em]
					     - \beta \partial_{\nu_\Sigma} v & = & 0   & \quad \mbox{on} & (0,\,T) \times \partial \Sigma, \\[0.5em]
			                                    v(0) & = & v_0 & \quad \mbox{on} & \Sigma.
				\end{array}\right.
			\end{equation*}
			If $f,\,v_0 \geq 0$, then $v \geq 0$.
	\end{enumerate}
\end{lemma}
\begin{proof}
	The proof follows the lines of the proof of \Lemref{Nonneg},
	except that here we deal with a linear problem, only.
\end{proof}
\begin{lemma}\lemlabel{Weak-Estimates-Domain}
	Let $T^\ast > 0$ and let $2 \leq p < \infty$ and $2 \leq q \leq \infty$.
	Let $\alpha,\,\beta > 0$.
	Assume $f \in \bF^\Omega_p(T)$, $g^{\textrm{in}} \in \bG^{\textrm{in}}_p(T)$, $\gs \in \bG^\Sigma_p(T)$
	and $v_0 \in \bI^\Omega_p \cap BC(\Omega)$ with
	\begin{equation*}
		(u \cdot \nu) v_0 - \beta \partial_\nu v_0 = g^{\textrm{in}}(0) \quad \mbox{on}\ \Gamma_{\textrm{in}}, \qquad
		- \beta \partial_\nu v_0 = \gs(0) \quad \mbox{on}\ \Sigma,
	\end{equation*}
	if $p > 3$.
	Let $v \in \bE^\Omega_p(T)$ be a strong solution to
	\begin{equation*}
		\left\{\begin{array}{rclll}
			\alpha \partial_t v + (u \cdot \nabla) v - \beta \Delta v & = & f               & \quad \mbox{in} & (0,\,T) \times \Omega,                \\[0.5em]
		                     (u \cdot \nu) v - \beta \partial_\nu v & = & g^{\textrm{in}} & \quad \mbox{on} & (0,\,T) \times \Gamma_{\textrm{in}},  \\[0.5em]
		                                     - \beta \partial_\nu v & = & \gs        & \quad \mbox{on} & (0,\,T) \times \Sigma,                \\[0.5em]
		                                     - \beta \partial_\nu v & = & 0               & \quad \mbox{on} & (0,\,T) \times \Gamma_{\textrm{out}}, \\[0.5em]
			                                                     v(0) & = & v_0             & \quad \mbox{in} & \Omega.
		\end{array}\right.
	\end{equation*}
	Then
	\begin{equation*}
		\|v\|_{L^q(\Omega_T)} + \|v\|_{L^q(\Sigma_T)} \leq C \Big( \|f\|_{L^q(\Omega_T)} + \|g^{\textrm{in}}\|_{L^q(\Gamma_{\textrm{in},T})} + \|\gs\|_{L^q(\Sigma_T)} + \|v_0\|_{L^q(\Omega)} \Big)
	\end{equation*}
	for some constant $C> 0$ that is independent of $p$ and $0 < T < T^\ast$.
\end{lemma}
\begin{proof} If the right-hand side is infinite, nothing has to be proved. So, we now assume that the data admits finite $L^q$-norm. We distinguish two cases:\\
{\it Case 1:} $p>5/2$, such that $\E^\Om_p(T)\hookrightarrow BC(\overline{\Om_T})$, cf.\ \cite{Amann-Anisotropic-Function-Spaces}. \\
{\it $q=2$:} standard (multiply the equation by $v$ to obtain
$L^2$-estimates).\\
{\it $q=\infty$:} Let $\cL_T$ denote the operator given by the left-hand side of the system under consideration, such that $\cL_T v=(f,g^\text{in},\gs,0,v_0)$. For fixed $5/2<r<3$ let $\phi\in \E^\Om_r(T)$ denote the solution of
\begin{align*}
\cL_T\phi =(1,-1,-1,0,1)=:F.
\end{align*} Note that $\E^\Om_r(T)\hookrightarrow BC(\overline{\Om_T})$ and that no compatibility conditions for $F$ occur. Let
\begin{align*}
\delta:=\Vert f\Vert_{L^\infty(\Om_T)}+\Vert g^\text{in}\Vert_{L^\infty(\Ga_{\text{in},T})}+\Vert \gs\Vert_{L^\infty(\Si_T)}+\Vert v_0\Vert_{L^\infty(\Om)}
\end{align*} and let $\bar{v}\in \E^\Om_p(T)$ be given through $\cL_T\bar{v}=\delta F$. Then the comparison principle (\Lemref{Linear-Positivity} a)) applied to $\bar{v}-v$ yields $v\leq \bar{v}$ on $\overline{\Om_T}$.
Since $\bar{v}=\delta \phi$ and $\Vert\bar{v}\Vert_{BC}=\delta \Vert \phi \Vert_{BC}$ it follows that
\begin{align*}
v\leq \delta \left(\Vert \phi \Vert_{L^\infty(\Om_T)}+\Vert \phi \Vert_{L^\infty(\Si_T)}\right)=:\delta C
\end{align*} on $\overline{\Om_T}$. Analogously $-\delta C\leq v$, such that $\Vert v \Vert_{L^\infty(\Om_T)}\leq C\delta$.
Since in this case $v\in BC(\overline{\Om_T})$, we trivially have
\begin{align*}
\Vert v\Vert_{L^\infty(\Si_T)}\leq \Vert v \Vert_{L^\infty(\Om_T)}\leq C\delta,
\end{align*} hence that
\begin{align*}
\Vert v\Vert_{L^\infty(\Si_T)}+\Vert v \Vert_{L^\infty(\Om_T)}\leq C\delta.
\end{align*}
For $2\leq q\leq\infty$ we set
\begin{align*}
X_q:=L^q(\Om_T)\times L^q(\Si_T), \quad Y_q:=L^q(\Om_T)\times L^q(\Ga_{\text{in},T})\times L^q(\Si_T)\times \{0\}\times L^q(\Om).
\end{align*} Let $\cS_T$ denote the system's solution operator. By the $L^2$- and the $L^\infty$-estimates obtained above we have
\begin{align*}
\cS_T \in \mathscr{L}(Y_2,X_2)\cap \mathscr{L}(Y_\infty,X_\infty).
\end{align*} By interpolation $\cS_T\in \mathscr{L}(Y_q,X_q)$ which yields the assertion for Case 1.

\medskip
{\it Case 2:} $2\leq p\leq 5/2$.\\
For convenience set $F:=(f,0,g^\text{in},\gs,0,0,v_0,0)\in \F^{\Om,\Si}_{p,I}(T)$. Let $r>5/2$ and choose
\begin{align*}
F_n:=&(f_n,0,g^\text{in}_n,\gs_n,0,0,v_{0,n},0) \in \mathbb{Y}_{r,q}:= \\
\F^{\Om,\Si}_{r,I}(T)\cap
(L^q(\Om_T)\times\{0\}\times &L^q(\Ga_{\text{in},T})\times L^q(\Si_T)\times\{0\}\times\{0\}\times L^q(\Om)\times\{0\})
\end{align*} for $n\in \N$, such that $F_n\rightarrow F$ in $\mathbb{Y}_{p,q}$ (e.g.\ by extension of $F$ to $\R\times\R^3$ respectively $\R\times (\pa A \times \R)$ and mollification). Let $v_n\in \E^\Om_r(T)$ denote the corresponding solution of $\cL_T v_n=F_n$. Then Case 1 applies to $F_n,v_n$ and there exists a $C>0$ independent of $n\in\N$, such that
\begin{align}\label{eq:Lq-estimate-approximation}
\Vert v_n \Vert_{L^q(\Om_T)}+\Vert v_n\Vert_{L^q(\Si_T)} \leq C \Vert F_n\Vert_{L^q}.
\end{align} Obviously $(v_n)_n$ is a Cauchy sequence in
\begin{align*}
\mathbb{X}_{p,q}:=\E^\Om_p(T)\cap \left(L^q(\Om_T)\times
L^q(\Si_T)\right),
\end{align*}
such that we may pass to the limit $n\rightarrow \infty$  in (\ref{eq:Lq-estimate-approximation}). Hence we obtain $v_n\rightarrow v$ in $\mathbb{X}_{p,q}$ with $v$ being the solution of $\cL_T v=F$ and
 \begin{align*}
\Vert v\Vert_{L^q(\Om_T)}+\Vert v\Vert_{L^q(\Si_T)} \leq C \Vert F\Vert_{L^q}.
\end{align*}
\end{proof}

The next Lemma is standard for equations on standard domains $\Om$, cf. \cite[Lemma 3.4]{Pie10}. Here we give a proof since we employ it on $\Si$.
\begin{lemma}\lemlabel{Weak-Estimates-Boundary}
	Let $T^\ast >0$ and let $1 <p,\,q<\infty$.
	Let $\mu > 0$ and let arbitrary coefficients $\alpha_1,\,\dots,\,\alpha_N,\,\beta_1,\,\dots,\,\beta_N \in \bR$ be given.
	Assume the data satisfies $f,\,g_1,\,\dots,\,g_n \in \bF^\Sigma_p(T)$ and $u_0,\,v^{1}_0,\,\dots,\,v^{N}_0 \in \bI^\Sigma_p \cap BC(\Sigma)$
	with $\partial_{\nu_\Sigma} u_0 = \partial_{\nu_\Sigma} v^{j}_0 = 0$ on $\partial \Sigma$, if $p > 3$.
	Furthermore, let $u,\,v_1,\,\dots,\,v_N \in \bE^\Sigma_p(T)$ be strong solutions to
	\begin{equation}\label{eq:duality-Lm-forward-heat-equation}
		\left\{\begin{array}{rclrclll}
			\partial_t u - \mu \Delta_\Si u & = & f,   & \qquad \alpha_j \partial_t v_j - \beta_j \Delta_\Si v_j & = & g_j       & \qquad \mbox{on} & (0,\,T) \times \Sigma,          \\[0.5em]
			  - \mu \partial_{\nu_\Sigma} u & = & 0,   & \qquad              - \beta_j \partial_{\nu_\Sigma} v_j & = & 0         & \qquad \mbox{on} & (0,\,T) \times \partial \Sigma, \\[0.5em]
			                           u(0) & = & u_0, & \qquad                                           v_j(0) & = & v^{j}_0 & \qquad \mbox{on} & \Sigma
		\end{array}\right.
	\end{equation}
	for some $0 < T < T^\ast$.
	If $f \leq \sum^N_{j = 1} g_j$, then 
\begin{equation*} \|u^+\|_{L^q(\Sigma_T)} \leq C \left(1 + \sum^N_{j = 1} \|v_j\|_{L^q(\Sigma_T)}\right)
\end{equation*}
 for some constant $C = C(\|u_0\|_{BC(\Sigma)},\,\|v^{1}_0\|_{BC(\Sigma)},\,\dots,\,\|v^{N}_0\|_{BC(\Sigma)}) > 0$
	that is independent of $0 < T < T^\ast$.
\end{lemma}
\begin{proof}
Let $\theta \in L^{q^\prime}(\Sigma_T)^+$ with $1/q+1/q'=1$ and let $0<\tau<T$. By the transformation $t\mapsto \tau -t$ applied to the equation for $u$ in (\ref{eq:duality-Lm-forward-heat-equation}) with data $f=\theta$, $u_0=0$ we obtain the backward heat equation,
\begin{align}\label{eq:backward-heat-equation}
\left\{\begin{array}{rcll}
-[\pa_t \phi +\mu \Delta_\Sigma \phi]&=&\theta(\tau-t)&\text{on}\; (0,\tau)\times\Si,\\
-\mu \pa_{\ns}\phi&=&0&\text{on}\;(0,\tau)\times\pa\Si,\\
\phi(\tau)&=&0&\text{on}\;\Si,
 \end{array}\right.
\end{align}
which admits maximal $L^{q'}$-regularity, cf.\ Section \ref{sec:linear_equations}. In particular, we have $\phi \in BUC([0,\tau],L^{q'}(\Si))$ and
\begin{align*}
\sup_{s\in[0,\tau]}\Vert \phi(s)\Vert_{L^{q'}(\Si)}+\Vert \pa_t\phi\Vert_{L^{q'}(\Si_\tau)}+\mu \Vert \Delta_\Si\phi \Vert_{L^{q'}(\Si_\tau)} \leq C \Vert \theta\Vert_{L^{q'}(\Si_\tau)}
\end{align*} for a constant $C>0$ which is independent of $\tau$. Observe that $\phi\geq 0$, since $\theta\geq 0$. Hence by plugging in the first line of (\ref{eq:backward-heat-equation}) and multiple partial integrations we have
\begin{align*}
\int_{\Si_\tau} u\theta d(t,\si) &=\int_\Si u_0\phi(0) d\si + \int_{\Si_\tau}\phi(\pa_t u-\mu \Delta_\Si u) d(t,\si)  \\
&\leq \int_\Si u_0\phi(0) d\si +\sum_{j=1}^N \int_{\Si_\tau} \phi (\alpha_j\pa_t v_j -\beta_j \Delta_\Si v_j) d(t,\si) \\
&=\int_\Si (u_0-\sum_{j=1}^N\alpha_j v^{j}_0)\phi(0) d\si + \sum_{j=1}^N \int_{\Si_\tau}\left(-\alpha_j\pa_t\phi -\beta_j \Delta_\Si \phi\right)v_j d(t,\si) \\
&\leq C \Vert \theta\Vert_{L^{q'}(\Si_\tau)}\left(1+\sum_{j=1}^N \Vert v_j\Vert_{L^q(\Si_\tau)}\right).
\end{align*} Employing
\begin{equation*}
		\|u^+\|_{L^q(\Sigma_T)} = \sup \left\{\ \int^T_0 \int_\Sigma u \theta\,\mbox{d}\sigma(x)\,\mbox{d}t\,:\,\theta \in L^{q^\prime}(\Sigma_T)^+,\ \|\theta\|_{L^{q^\prime}(\Sigma_T)} < 1\ \right\}
	\end{equation*}
	the assertion follows.
\end{proof}

Now, we are in position to prove \Thmref{Global-WP}.
\begin{proof}[Proof of \Thmref{Global-WP}]
	Since according to \Thmref{Local-WP} the (local-in-time) solutions to (\ref{eq:cat}) generate a local semi-flow in the
	phase space $\bI^\Omega_2 \times \bI^\Sigma_2 = H^1(\Omega)
	\times H^1(\Sigma)$,
	we may assume $T^\ast < \infty$ and show that the solution stays bounded in $H^1(\Omega) \times H^1(\Sigma)$ on $(0,\,T^\ast)$
	in order to obtain a contradiction.
	It is sufficient to establish $L^\infty$-bounds for the solution in order to obtain boundedness in the phase space.
Then the $H^1$-boundedness of solutions follows, as it is shown in the last part of this proof.

	We will now derive $L^\infty$-bounds, which requires several steps.
	Note that we may use the fact that $c_i,\,\cs_i \geq 0$ on $(0,\,T^\ast)$ thanks to \Lemref{Nonneg}.

	{\bfseries Step 1.}
	We have $c_i \in \bE^\Omega_2(T) = H^1((0,\,T),\,L^2(\Omega)) \cap L^2((0,\,T),\,H^2(\Omega))$ and
	\begin{equation}
		\eqnlabel{Global-WP-Base-Domain}
		\left\{\begin{array}{rclll}
			\partial_t c_i + (u \cdot \nabla) c_i - d_i \Delta c_i & = & f_i                                   & \quad \mbox{in} & (0,\,T) \times \Omega,                \\[0.5em]
		                (u \cdot \nu) c_i - d_i \partial_\nu c_i & = & g^{\textrm{in}}_i                     & \quad \mbox{on} & (0,\,T) \times \Gamma_{\textrm{in}},  \\[0.5em]
		                                  - d_i \partial_\nu c_i & = & r^{\textrm{sorp}}_i(c_i,\,\cs_i) & \quad \mbox{on} & (0,\,T) \times \Sigma,                \\[0.5em]
		                                  - d_i \partial_\nu c_i & = & 0                                     & \quad \mbox{on} & (0,\,T) \times \Gamma_{\textrm{out}}, \\[0.5em]
			                                                c_i(0) & = & c_{0,i}                               & \quad \mbox{in} & \Omega
		\end{array}\right.
	\end{equation}
	for $i = 1,\,\dots,\,N$ and all $0 < T < T^\ast$.
	Now, $r^{\textrm{sorp}}_i(c_i,\,\cs_i) \geq - k^{\textrm{de}}_i \cs_i$.
	Thus, by \Lemref{Linear-Positivity}~a) we have $0 \leq c_i \leq z_i$ for the unique maximal regular solution $z_i$ to
	\begin{equation*}
		\left\{\begin{array}{rclll}
			\partial_t z_i + (u \cdot \nabla) z_i - d_i \Delta z_i & = & f_i               & \quad \mbox{in} & (0,\,T) \times \Omega,                \\[0.5em]
		                (u \cdot \nu) z_i - d_i \partial_\nu z_i & = & g^{\textrm{in}}_i & \quad \mbox{on} & (0,\,T) \times \Gamma_{\textrm{in}},  \\[0.5em]
		                                  - d_i \partial_\nu z_i & = &-C \cs_i      & \quad \mbox{on} & (0,\,T) \times \Sigma,                \\[0.5em]
		                                  - d_i \partial_\nu z_i & = & 0                 & \quad \mbox{on} & (0,\,T) \times \Gamma_{\textrm{out}}, \\[0.5em]
			                                                z_i(0) & = & c_{0,i}           & \quad \mbox{in} & \Omega
		\end{array}\right.
	\end{equation*}
	with some appropriate constant $C = C((k^{\textrm{de}}_j)_{j = 1, \dots, N}) > 0$.
	Note that this problem allows for a strong solution in the $L^2$-setting without any compatibility conditions
	between the right hand sides of the boundary conditions and the initial value.
	Since $z_i$ is a solution to a linear problem, we may write $z_i = u_i + v_i$,
	where $u_i$ solves
	\begin{equation*}
		\left\{\begin{array}{rclll}
			\partial_t u_i + (u \cdot \nabla) u_i - d_i \Delta u_i & = & 0            & \quad \mbox{in} & (0,\,T) \times \Omega,                \\[0.5em]
		                (u \cdot \nu) u_i - d_i \partial_\nu u_i & = & 0            & \quad \mbox{on} & (0,\,T) \times \Gamma_{\textrm{in}},  \\[0.5em]
		                                  - d_i \partial_\nu u_i & = & -C \cs_i & \quad \mbox{on} & (0,\,T) \times \Sigma,                \\[0.5em]
		                                  - d_i \partial_\nu u_i & = & 0            & \quad \mbox{on} & (0,\,T) \times \Gamma_{\textrm{out}}, \\[0.5em]
			                                                u_i(0) & = & 0            & \quad \mbox{in} & \Omega
		\end{array}\right.
	\end{equation*}
	and $v_i$ solves
	\begin{equation*}
		\left\{\begin{array}{rclll}
			\partial_t v_i + (u \cdot \nabla) v_i - d_i \Delta v_i & = & f_i               & \quad \mbox{in} & (0,\,T) \times \Omega,                \\[0.5em]
		                (u \cdot \nu) v_i - d_i \partial_\nu v_i & = & g^{\textrm{in}}_i & \quad \mbox{on} & (0,\,T) \times \Gamma_{\textrm{in}},  \\[0.5em]
		                                  - d_i \partial_\nu v_i & = & 0                 & \quad \mbox{on} & (0,\,T) \times \Sigma,                \\[0.5em]
		                                  - d_i \partial_\nu v_i & = & 0                 & \quad \mbox{on} & (0,\,T) \times \Gamma_{\textrm{out}}, \\[0.5em]
			                                                v_i(0) & = & c_{0,i}            & \quad \mbox{in} & \Omega
		\end{array}\right.
	\end{equation*}
	For these solutions we have
	\begin{equation*}
		\|u_i\|_{L^q(\Omega_T)} + \|u_i\|_{L^q(\Sigma_T)} \leq C^\prime \|\cs_i\|_{L^q(\Sigma_T)}
	\end{equation*}
	as well as
	\begin{equation*}
		\|v_i\|_{L^q(\Omega_T)} + \|v_i\|_{L^q(\Sigma_T)} \leq A_i
	\end{equation*}
	provided that $2 \leq q \leq\infty$.
	Here, we employed \Lemref{Weak-Estimates-Domain} to obtain constants $C^\prime = C^\prime(C,\,q) > 0$
	and $A_i = A_i(\|f_i\|_{L^q(\Omega_T)},\,\|g^{\textrm{in}}_i\|_{L^q(\Gamma_{\textrm{in},T})},\,\|c_{0,i}\|_{L^q(\Omega)},\,q) > 0$
	that are independent of $0 < T < T^\ast$.
	Since $\|c_i\|_{L^q} \leq \|z_i\|_{L^q} \leq \|u_i\|_{L^q} + \|v_i\|_{L^q}$ we may sum up the above estimates to obtain
	\begin{equation}
		\eqnlabel{Global-WP-Estimates-Domain}
		\sum^N_{i = 1} \|c_i\|_{L^q(\Omega_T)} + \sum^N_{i = 1} \|c_i\|_{L^q(\Sigma_T)} \leq C^\ast \left( 1 + \sum^N_{j = 1} \|\cs_j\|_{L^q(\Sigma_T)} \right),
	\end{equation}
	where $C^\ast = C^\ast(C^\prime, (A_j)_{j = 1, \dots, N}) > 0$ denotes a constant that is independent of $0 < T < T^\ast$.
	Note that this estimate is available for all $2 \leq q \leq \infty$.

	{\bfseries Step 2.}
	We have $\cs_i \in \bE^\Sigma_2(T) = H^1((0,\,T),\,L^2(\Sigma)) \cap L^2((0,\,T),\,H^2(\Sigma))$ and
	\begin{equation}
		\eqnlabel{Global-WP-Base-Boundary}
		\left\{\begin{array}{rclll}
			\partial_t \cs_i - \ds_i \Delta_\Sigma \cs_i & = & r^{\textrm{sorp}}_i(c_i,\,\cs_i) + r^{\textrm{ch}}_i(\cs) & \quad \mbox{on} & (0,\,T) \times \Sigma,          \\[0.5em]
			              - \ds_i \partial_{\nu_\Sigma} \cs_i & = & 0                                                                   & \quad \mbox{on} & (0,\,T) \times \partial \Sigma, \\[0.5em]
			                                              \cs_i(0) & = & \cs_{0,i}                                                       & \quad \mbox{on} & \Sigma			
		\end{array}\right.
	\end{equation}
	for $i = 1,\,\dots,\,N$ and all $0 < T < T^\ast$.
	Now we use the triangular structure of the reaction rates that is guaranteed by ($\textrm{A}^{\textrm{ch}}_{\textrm{S}}$) to
	treat the cases $i = 1$ and $i = 2,\,\dots,\,N$ separately.

	{\bfseries Step 2.1.}
	According to assumption ($\textrm{A}^{\textrm{sorp}}_{\textrm{B}}$) we have $r^{\textrm{sorp}}_1(c_1,\,\cs_1) \leq k^{\textrm{ad}}_1 c_1$
	and according to ($\textrm{A}^{\textrm{ch}}_{\textrm{S}}$) we have $q_{11} r^{\textrm{ch}}_1(\cs) \leq C (1 + \sum^N_{j = 1} \cs_j)$
	for some fixed constant $C > 0$.
	Thus, by \Lemref{Linear-Positivity}~b) we have $0 \leq \cs_1 \leq \zs_1$ for the unique solution $\zs_1$ to
	\begin{equation*}
		\left\{\begin{array}{rclll}
			\partial_t \zs_1 - \ds_1 \Delta_\Sigma \zs_1 & = & C^\prime \Big( 1 + c_1 + \sum^N_{j = 1} \cs_j \Big) & \quad \mbox{on} & (0,\,T) \times \Sigma,          \\[0.5em]
			              - \ds_1 \partial_{\nu_\Sigma} \zs_1 & = & 0                                                        & \quad \mbox{on} & (0,\,T) \times \partial \Sigma, \\[0.5em]
			                                              \zs_1(0) & = & \cs_{0,1}                                            & \quad \mbox{on} & \Sigma
		\end{array}\right.		
	\end{equation*}
	with some appropriate constant $C^\prime = C^\prime(C,\,k^{\textrm{ad}}_1,\,q_{11}) > 0$.
	Since $\zs_1$ is a solution to a linear problem, we may write $\zs_1 = \us_1 + \vs_1$,
	where $\us_1$ solves
	\begin{equation*}
		\left\{\begin{array}{rclll}
			\partial_t \us_1 - \ds_1 \Delta_\Sigma \us_1 & = & C^\prime \Big( 1 + c_1 + \sum^N_{j = 1} \cs_j \Big) & \quad \mbox{on} & (0,\,T) \times \Sigma,          \\[0.5em]
			              - \ds_1 \partial_{\nu_\Sigma} \us_1 & = & 0                                                        & \quad \mbox{on} & (0,\,T) \times \partial \Sigma, \\[0.5em]
			                                              \us_1(0) & = & 0                                                        & \quad \mbox{on} & \Sigma
		\end{array}\right.		
	\end{equation*}
	and $\vs_1$ solves
	\begin{equation*}
		\left\{\begin{array}{rclll}
			\partial_t \vs_1 - \ds_1 \Delta_\Sigma \vs_1 & = & 0             & \quad \mbox{on} & (0,\,T) \times \Sigma,          \\[0.5em]
			              - \ds_1 \partial_{\nu_\Sigma} \vs_1 & = & 0             & \quad \mbox{on} & (0,\,T) \times \partial \Sigma, \\[0.5em]
			                                              \vs_1(0) & = & \cs_{0,1} & \quad \mbox{on} & \Sigma.
		\end{array}\right.		
	\end{equation*}
	For these solutions we have
	\begin{equation*}
		\|\us_1\|_{L^q(\Sigma_T)} \leq M \|\us_1\|_{{}_0 \bE^\Sigma_p(T)} \leq L M C^\prime \left( \Theta + \|c_1\|_{L^p(\Sigma_T)} + \sum^N_{j = 1} \|\cs_j\|_{L^p(\Sigma_T)} \right)
	\end{equation*}
	with $\Theta = \Theta(p) = \|1\|_{L^p(\Sigma_{T^\ast})}$ as well as
	\begin{equation*}
		\|(\vs_1)^+\|_{L^q(\Sigma_T)} \leq A_1,
	\end{equation*}
	provided that $2 \leq p,\,q < \infty$.
	Here, $M = M(p,\,q) > 0$ denotes the norm of the embedding ${}_0 \bE^\Sigma_p(T) \hookrightarrow L^q(\Sigma_T)$, cf.\ \cite{Amann-Anisotropic-Function-Spaces}
	and $L = L(p) > 0$ denotes the norm of the solution operator in the $L^p$-setting for the time interval $(0,\,T)$,
	which are both independent of $0 < T < T^\ast$ thanks to the homogeneous initial condition.
	Furthermore, $A_1 = A_1(\|\cs_{0, 1}\|_{BC(\Sigma)},\,q) > 0$ denotes the constant delivered by \Lemref{Weak-Estimates-Boundary},
	which is also independent of $0 < T < T^\ast$. Observe, that a standard maximum principle could have been applied here, too.
	Note that $\vs_1\leq (\vs_1)^+$ implies $0 \leq \cs_1\leq \zs_1 \leq \us_1 + (\vs_1)^+$.
	Therefore,
	\begin{equation*}
		\begin{array}{rcl}
			\|\cs_1\|_{L^q(\Sigma_T)} & \leq & \|\zs_1\|_{L^q(\Sigma_T)} \leq \|\us_1\|_{L^q(\Sigma_T)} + \|(\vs_1)^+\|_{L^q(\Sigma_T)} \\[1.0em]
				& \leq & C^{\prime\prime} {\displaystyle \left( 1 + \|c_1\|_{L^p(\Sigma_T)} + \sum^N_{j = 1} \|\cs_j\|_{L^p(\Sigma_T)} \right),}
		\end{array}
	\end{equation*}
	provided that $2 \leq p,\,q < \infty$.
	Here $C^{\prime\prime} = C^{\prime\prime}(L M C^\prime,\,\Theta,\,A_1) > 0$ is independent of $0 < T < T^\ast$.

	{\bfseries Step 2.2.}
	Now fix $i \in \{\,2,\,\dots,\,N\,\}$.
	By \eqnrefbr{Global-WP-Base-Boundary} we obtain
	\begin{equation*}
		q_{ii} (\partial_t \cs_i - \ds_i \Delta_\Sigma \cs_i) + \sum_{j < i} q_{ij} \big( \partial_t \cs_j - \ds_j \Delta_\Sigma \cs_j \big) = [Q r^{\textrm{sorp}}(c,\,\cs)]_i + [Q r^{\textrm{ch}}(\,\cs)]_i
	\end{equation*}
	on $(0,\,T) \times \Sigma$.
	According to assumption ($\textrm{A}^{\textrm{sorp}}_{\textrm{B}}$) we have
	\begin{equation*}
		[Q r^{\textrm{sorp}}(c,\,\cs)]_i = \sum_{j \leq i}q_{ij} r^{\textrm{sorp}}_j(c_j,\,\cs_j) \leq \sum_{j \leq i, q_{ij}>0}q_{ij}k^\text{ad}_jc_j - \sum_{j \leq i, q_{ij}<0}q_{ij}k^\text{de}_j\cs_j
	\end{equation*}
	and according to ($\textrm{A}^{\textrm{ch}}_{\textrm{S}}$) we have $[Q r^{\textrm{ch}}(\,\cs)]_i \leq C (1 + \sum^N_{j = 1} \cs_j)$
	for some fixed constant $C > 0$.
	Thus, by \Lemref{Linear-Positivity}~b) we have $0 \leq \cs_i \leq \zs_i$ for the unique maximal regular solution $\zs_i$ to
	\begin{equation*}
		\left\{\begin{array}{rclll}
			\partial_t \zs_i - \ds_i \Delta_\Sigma \zs_i & = & C^\prime \Big( 1 + \sum_{j \leq i} c_j + \sum^N_{j = 1} \cs_j \Big) \\[1.0em]
			                                                            &   & \quad - \sum_{j < i} r_{ij} \big( \partial_t \cs_j - \ds_j \Delta_\Sigma \cs_j \big) & \quad \mbox{in} & (0,\,T) \times \Sigma,          \\[0.5em]
			              - \ds_i \partial_{\nu_\Sigma} \zs_i & = & 0                                                                                                   & \quad \mbox{on} & (0,\,T) \times \partial \Sigma, \\[0.5em]
			                                              \zs_i(0) & = & \cs_{0,i}                                                                                       & \quad \mbox{on} & \Sigma
		\end{array}\right.		
	\end{equation*}
	with some appropriate constant $C^\prime = C^\prime(C,\,k^{\textrm{ad}}_j,\,k^\text{de}_j,\,q_{ij}) > 0$ and $r_{ij} = q_{ij} / q_{ii}$.
	Since $\zs_i$ is a solution to a linear problem, we may write $\zs_i = \us_i + \vs_i$,
	where $\us_i$ solves
	\begin{equation*}
		\left\{\begin{array}{rclll}
			\partial_t \us_i - \ds_i \Delta_\Sigma \us_i & = & C^\prime \Big( 1 + \sum_{j \leq i} c_j + \sum^N_{j = 1} \cs_j \Big) & \quad \mbox{in} & (0,\,T) \times \Sigma,          \\[0.5em]
			              - \ds_i \partial_{\nu_\Sigma} \us_i & = & 0                                                                        & \quad \mbox{on} & (0,\,T) \times \partial \Sigma, \\[0.5em]
			                                              \us_i(0) & = & 0                                                                        & \quad \mbox{on} & \Sigma
		\end{array}\right.		
	\end{equation*}
	and $\vs_i$ solves
	\begin{equation*}
		\left\{\begin{array}{rclll}
			\partial_t \vs_i - \ds_i \Delta_\Sigma \vs_i & = & - \sum_{j < i} r_{ij} \big( \partial_t \cs_j - \ds_j \Delta_\Sigma \cs_j \big) & \quad \mbox{in} & (0,\,T) \times \Sigma,          \\[0.5em]
			              - \ds_i \partial_{\nu_\Sigma} \vs_i & = & 0                                                                                             & \quad \mbox{on} & (0,\,T) \times \partial \Sigma, \\[0.5em]
			                                              \vs_i(0) & = & \cs_{0,i}                                                                                 & \quad \mbox{on} & \Sigma.
		\end{array}\right.		
	\end{equation*}
	For these solutions we have
	\begin{equation*}
		\|\us_i\|_{L^q(\Sigma_T)} \leq M \|\us_i\|_{{}_0 \bE^\Sigma_p(T)} \leq L M C^\prime \left( \Theta + \sum_{j \leq i} \|c_j\|_{L^p(\Sigma_T)} + \sum^N_{j = 1} \|\cs_j\|_{L^p(\Sigma_T)} \right)
	\end{equation*}
	with $\Theta = \Theta(p) = \|1\|_{L^p(\Sigma_{T^\ast})}$ as well as
	\begin{equation*}
		\|(\vs_i)^+\|_{L^q(\Sigma_T)} \leq A_i \left( 1 + \sum_{j < i} \|\cs_j\|_{L^q(\Sigma_T)} \right),
	\end{equation*}
	provided that $2 \leq p,\,q < \infty$.
	Here, $M = M(p,\,q) > 0$ denotes the norm of the embedding ${}_0 \bE^\Sigma_p(T) \hookrightarrow L^q(\Sigma_T)$
	and $L = L(p) > 0$ denotes the norm of the solution operator in the $L^p$-setting for the time interval $(0,\,T)$,
	which are both independent of $0 < T < T^\ast$ thanks to the homogeneous initial condition as in Step 2.1.
	Furthermore, $A_i = A_i((\|\cs_{0, j}\|_{BC(\Sigma)})_{j = 1, \dots, i},\,q) > 0$ denotes the constant delivered by \Lemref{Weak-Estimates-Boundary},
	which is also independent of $0 < T < T^\ast$.
	We again have $\cs_i \leq \zs_i \leq \us_i + (\vs_i)^+$.
	Therefore,
	\begin{equation*}
		\begin{array}{rcl}
			\|\cs_i\|_{L^q(\Sigma_T)} & \leq & \|\zs_i\|_{L^q(\Sigma_T)} \leq \|\us_i\|_{L^q(\Sigma_T)} + \|(\vs_i)^+\|_{L^q(\Sigma_T)} \\[1.0em]
				& \leq & C^{\prime\prime} {\displaystyle \left( 1 + \sum_{j \leq i} \|c_j\|_{L^p(\Sigma_T)} + \sum^N_{j = 1} \|\cs_j\|_{L^p(\Sigma_T)} \right) + A_i \sum_{j < i} \|\cs_j\|_{L^q(\Sigma_T)},}
		\end{array}
	\end{equation*}
	provided that $2 \leq p,\,q < \infty$.
	Here $C^{\prime\prime} = C^{\prime\prime}(L M C^\prime,\,A_i,\,\Theta) > 0$ is independent of $0 < T < T^\ast$.

	{\bfseries Step 2.3.}
	Now we may combine the estimates obtained in Steps 2.1 and 2.2, recursively, and infer that
	\begin{equation}
		\eqnlabel{Global-WP-Estimates-Boundary}
		\sum^N_{i = 1} \|\cs_i\|_{L^q(\Sigma_T)} \leq C^\ast \left( 1 + \sum^N_{j = 1} \|c_j\|_{L^p(\Sigma_T)} + \sum^N_{j = 1} \|\cs_j\|_{L^p(\Sigma_T)} \right)
	\end{equation}
	provided that $2 \leq p,\,q < \infty$.
	Here, $C^\ast = C^\ast(C^{\prime\prime},\,(A_j)_{j = 1, \dots, N}) > 0$ is independent of $0 < T < T^\ast$.

	{\bfseries Step 3.}
	Now we combine estimates (\ref{eqn:Global-WP-Estimates-Domain}) and (\ref{eqn:Global-WP-Estimates-Boundary}) to obtain
	\begin{equation}
		\eqnlabel{Global-WP-Estimates-Final}
		\begin{array}{l}
			{\displaystyle \sum^N_{i = 1} \|c_i\|_{L^q(\Omega_T)} + \sum^N_{i = 1} \|c_i\|_{L^q(\Sigma_T)} + \sum^N_{i = 1} \|\cs_i\|_{L^q(\Sigma_T)}} \\[2.0em]
				\qquad \qquad \leq {\displaystyle C^\ast \left( 1 + \sum^N_{j = 1} \|c_j\|_{L^p(\Omega_T)} + \sum^N_{j = 1} \|c_j\|_{L^p(\Sigma_T)} + \sum^N_{j = 1} \|\cs_j\|_{L^p(\Sigma_T)} \right)},
		\end{array}
	\end{equation}
	provided that $2 \leq p,\,q < \infty$.
	Here, $C^\ast > 0$ is independent of $0 < T < T^\ast$.
	Using this inequality for $p = 2$ we may in particular obtain $L^q$-$L^2$-estimates for arbitrary $2 \leq q < \infty$.

	{\bfseries Step 4.1.}
	The surface concentrations satisfy
	\begin{equation*}
		\left\{\begin{array}{rclcl}
			\pa_t \cs_i - d^{\scriptscriptstyle\Si}_i \Delta_\Si \cs_i &=& \fs_i & \text{on}& (0,T)\times \Si, \\
			-\ds_i\pa_{\nu_\Si} \cs_i &=&0 &\text{on}& (0,T)\times \pa \Si, \\
			{\cs_i}|_{t=0}&=&\cs_{0,i} & \text{on}& \Si.
		\end{array}\right.
	\end{equation*}
	with $\fs_i := r^\text{sorp}_i(c_i,\cs_i)+r^\text{ch}_i(\cs)$.
	Due to the polynomial growth of the nonlinearities we may estimate $\fs$ in terms of $c$ and $\cs$
	and employ (\ref{eqn:Global-WP-Estimates-Final}) to obtain
	\begin{equation*}
		{\displaystyle \sum^N_{i = 1} \|\fs_i\|_{L^r(\Sigma_T)}}
			\leq {\displaystyle C^\ast \left( 1 + \sum^N_{j = 1} \|c_j\|_{L^p(\Omega_T)} + \sum^N_{j = 1} \|c_j\|_{L^p(\Sigma_T)} + \sum^N_{j = 1} \|\cs_j\|_{L^p(\Sigma_T)} \right)},
	\end{equation*}
	where $2 \leq p,\,r < \infty$ and $C^\ast > 0$ is independent of $0 < T < T^\ast$.
	Thus, for given $2 \leq q \leq \infty$ we may use this estimate for sufficiently large $2 \leq r < \infty$
	together with a classical result from \cite{LSU}, which yields the estimate
	\begin{equation}
		\eqnlabel{Global-WP-Estimates-Boundary-Infinity}
		{\displaystyle \sum^N_{i = 1} \|\cs_i\|_{L^q(\Sigma_T)}}
			\leq {\displaystyle C^\ast \left( 1 + \sum^N_{j = 1} \|c_j\|_{L^p(\Omega_T)} + \sum^N_{j = 1} \|c_j\|_{L^p(\Sigma_T)} + \sum^N_{j = 1} \|\cs_j\|_{L^p(\Sigma_T)} \right)}.
	\end{equation}
	Let us note that \cite[Theorem~III.7.1]{LSU} is stated for Dirichlet boundary conditions,
	but the result remains true in the Neumann case; see \cite[Theorem~4]{Bothe-Rolland:Global-Existence}, whose proof carries over to smooth manifolds as $\Sigma$.
	Note that in contrast to (\ref{eqn:Global-WP-Estimates-Boundary}) obtained in the second step,
	the estimate (\ref{eqn:Global-WP-Estimates-Boundary-Infinity}) is available for all $2 \leq p < \infty$ and all $2 \leq q \leq \infty$,
	while $C^\ast > 0$ is still independent of $0 < T < T^\ast$.

	{\bfseries Step 4.2.}
	Now we combine estimates (\ref{eqn:Global-WP-Estimates-Domain}) and (\ref{eqn:Global-WP-Estimates-Boundary-Infinity}) to obtain
	\begin{equation}
		\eqnlabel{Global-WP-Estimates-Final-Infinity}
		\begin{array}{l}
			{\displaystyle \sum^N_{i = 1} \|c_i\|_{L^q(\Omega_T)} + \sum^N_{i = 1} \|c_i\|_{L^q(\Sigma_T)} + \sum^N_{i = 1} \|\cs_i\|_{L^q(\Sigma_T)}} \\[2.0em]
				\qquad \qquad \leq {\displaystyle C^\ast \left( 1 + \sum^N_{j = 1} \|c_j\|_{L^p(\Omega_T)} + \sum^N_{j = 1} \|c_j\|_{L^p(\Sigma_T)} + \sum^N_{j = 1} \|\cs_j\|_{L^p(\Sigma_T)} \right)},
		\end{array}
	\end{equation}
	provided that $2 \leq p < \infty$ and $2 \leq q \leq \infty$.
	Here, $C^\ast > 0$ is independent of $0 < T < T^\ast$.
	Using this inequality for $p = 2$ we may in particular obtain $L^q$-$L^2$-estimates for arbitrary $2 \leq q \leq \infty$.

	{\bfseries Step 5.1.}
	The estimate (\ref{eqn:Global-WP-Estimates-Final-Infinity}) applied, for $p = 2$ and $q = \infty$, yields
	\begin{equation*}
		\begin{array}{l}
			{\displaystyle \sum^N_{i = 1} \|c_i(t)\|_{L^2(\Omega)}^2 + \sum^N_{i = 1} \|c_i(t)\|_{L^2(\Sigma)}^2 + \sum^N_{i = 1} \|\cs_i(t)\|_{L^2(\Sigma)}^2} \\[2.0em]
				\qquad \qquad \leq {\displaystyle C^\ast \left( 1 + \sum^N_{j = 1} \int^t_0 \|c_j(s)\|_{L^2(\Omega)}^2\,\mbox{d}s + \sum^N_{j = 1} \int^t_0 \|c_j(s)\|_{L^2(\Sigma)}^2\,\mbox{d}s + \sum^N_{j = 1} \int^t_0 \|\cs_j(s)\|_{L^2(\Sigma)}^2\,\mbox{d}s \right)}
		\end{array}
	\end{equation*}
	for all $0 < t < T < T^\ast$, where $C^\ast > 0$ is independent of $0 < T < T^\ast$.
	Thus, a standard Gronwall argument implies
	\begin{equation}
		\eqnlabel{Global-WP-Estimates-Gronwall}
		\|c_i\|_{L^2(\Omega_T)},\ \|c_i\|_{L^2(\Sigma_T)},\ \|\cs_i\|_{L^2(\Sigma_T)} \leq M e^{\omega T}, \qquad \qquad 0 < T < T^\ast,
	\end{equation}
	with some constants $M,\,\omega > 0$, which are independent of $0 < T < T^\ast$.

	{\bfseries Step 5.2.}
	The estimate (\ref{eqn:Global-WP-Estimates-Final-Infinity}) again applied for $p = 2$ and $q = \infty$ together with (\ref{eqn:Global-WP-Estimates-Gronwall}) yields
	\begin{equation}
		\eqnlabel{Global-WP-Estimates-Final-Gronwall}
		\|c_i\|_{L^\infty(\Omega_T)},\ \|c_i\|_{L^\infty(\Sigma_T)},\ \|\cs_i\|_{L^\infty(\Sigma_T)} \leq M e^{\omega T}, \qquad \qquad 0 < T < T^\ast,
	\end{equation}
	with some constants $M,\,\omega > 0$, which are independent of $0 < T < T^\ast$.

{\bfseries Step 6.}
Now the obtained a priori estimates (\ref{eqn:Global-WP-Estimates-Final-Gronwall}) carry over from $L^\infty$ to $H^1(\Omega)$ and $H^1(\Sigma)$. This may be seen by the following argument: Due to the $L^\infty$-estimates, the $L^2$-solution of (1) is contained in $\mathbb{E}^\Omega_p(T)\times \mathbb{E}^\Sigma_p(T)$ for each $1<p<\infty$ with $p\neq 3$ by bootstrapping. Here the crucial
estimate is
\begin{equation*}
\Vert r^\text{sorp}_i(c_i,\cs_i)\Vert_{\G^\Si_p(T)}\leq C \Vert (c_i,\cs_i)\Vert_{\E^\Om_q(T)\times \E^\Si_q(T)}
\end{equation*} with suitable $q<p$; see the proof of \Thmref{Local-WP}. This, in turn, yields
\begin{equation*}
c_i\in BC(\overline{\Omega_T}), \qquad \cs_i\in BC(\overline{\Sigma_T}), \qquad (T<T^\ast).
\end{equation*} Hence, by plugging in $c$, $\cs$ into $r^\text{sorp}$, $r^\text{ch}$, we may consider (\ref{eq:cat}) again, as a linear problem, this time for data being continuous in time.
More precisely, we consider (\ref{eq:withPerturbation}) for data
\begin{align*}
f_i&\in BC([0,T^*),L^2(\Omega)),\qquad \fs_i \in BC([0,T^*),L^2(\Sigma)),\\
g^\text{in}_i&\in BC([0,T^*),L^2(\Gamma_\text{in})), \quad \gs_i\in BC([0,T^*),L^2(\Sigma)), \\ g^\text{out}_i&\in BC([0,T^*),L^2(\Gamma_\text{out})),
\end{align*} and
\begin{align*}
c_{0,i}\in H^1(\Omega),\qquad \cs_{0,i}\in H^1(\Sigma).
\end{align*} Following the strategy of the proof of our linear result in Section \ref{sec:linear_equations},
in particular by transferring Lemma \ref{lm:Surjectivity-of-Trace} to handle the inhomogeneous boundary data $g^\text{in}_i$, $g^\text{out}_i$ and $\gs_i$
and by employing a semigroup representation, we obtain that the unique solution of (\ref{eq:cat}) satisfies
\begin{align*}
c_i\in BC([0,T^*),H^1(\Omega)), \qquad \cs_i\in BC([0,T^*),H^1(\Sigma)),
\end{align*} and the corresponding a priori estimates
\begin{align}\label{eq:final-BC-estimate}
\sup_{t\in[0,T]}(\Vert c_i(t)\Vert_{H^1(\Om)}+ \Vert \cs_i(t)\Vert_{H^1(\Si)}) \leq M e^{\omega T} \sup_{t\in[0,T]}\big( \Vert f_i(t)\Vert_{L^2(\Om)}
+ \Vert \fs_i(t)\Vert_{L^2(\Si)}\nonumber\\
+  \Vert g^\text{in}_i(t)\Vert_{L^2(\Ga_\text{in})} +\Vert g^\Si_i(t)\Vert_{L^2(\Si)}
+ \Vert g^\text{out}_i(t)\Vert_{L^2(\Ga_\text{out})}  \big), \qquad (T<T^*)
\end{align} for constants $M,\omega>0$ independent of $T$. Hence, we may pass to the limit $T\rightarrow T^\ast$
and see that both sides of (\ref{eq:final-BC-estimate}) stay finite.
The proof of \Thmref{Global-WP} is now complete.
\end{proof}

\begin{remark}[Sorption and reaction examples]\label{rk:revisiting-examples} 
A few remarks on the examples given in the introduction are in order here concerning the assumptions
($A_\text{F}^\text{sorp}$), ($A_\text{M}^\text{sorp}$), ($A_\text{B}^\text{sorp}$),
($A_\text{F}^\text{ch}$), ($A_\text{N}^\text{ch}$), ($A_\text{P}^\text{ch}$), ($A_\text{S}^\text{ch}$)
stated in Section \ref{sec:loc} and \Secref{Global-WP}. Evidently, Henry's law (S1) satisfies all of our assumptions. However, Langmuir's law (S2) needs to be modified in order to meet all assumptions on the sorption.
To this end we introduce $\zeta^+$, a smooth cut-off function, which approximates $(\cdot)^+$ pointwise and $\zeta^B$ a smooth and bounded function with bounded derivatives up to order $2$ which is monotonically increasing. In addition suppose $\zeta^+(0)=0$ and $\zeta^B(0)=0$. Then we consider
\begin{align*}
\tilde{r}^\text{sorp}_{L,i}(c_i,\cs_i)=k^\text{ad}_i\zeta^\text{B}(c_i)\zeta^+\left(1-\frac{\zeta^+(\cs_i)}{\cs_{\infty,i}}\right)-k^\text{de}_i\cs_i,
\end{align*} which indeed satisfies ($A_\text{F}^\text{sorp}$),
($A_\text{M}^\text{sorp}$), ($A_\text{B}^\text{sorp}$) and therefore is
covered by our main results. This modification is only necessary due to
technical reasons, since
\begin{itemize}
\item ($A_\text{F}^\text{sorp}$) is violated due to $\nabla r^\text{sorp}_{L,i} \notin BC^1(\R^2,\R^2)$.
\item ($A_\text{M}^\text{sorp}$) is not guaranteed since $\cs_i\leq \cs_{\infty, i}$ is not postulated.
\item For the same reason ($A_\text{B}^\text{sorp}$) is not satisfied since $(1-\cs_i/\cs_{\infty,i})$ could be negative.
\end{itemize}
For the time being it is not clear to the authors whether there are still global solutions in case we omit the cut-off functions. Observe that in our model there is no maximal capacity on the active surface, which 
is required in the original Langmuir law (S2) to gain nonnegativity of concentrations. Nonnegativity in turn is employed in the proof of the global existence result. 

The reaction rate $r^\text{ch}$ given in (R1) satisfies ($A^\text{ch}_\text{F}$), admits quadratic growth ($A^\text{ch}_\text{P}$), is quasi-positive ($A^\text{ch}_\text{N}$) and respects the triangular structure
($\textrm{A}^{\textrm{ch}}_{\textrm{S}}$) with corresponding matrix
\begin{align*}
Q=\begin{pmatrix}
1&0&0\\
0&1&0\\
1&0&1 \end{pmatrix}.
\end{align*}
\end{remark}

\bibliographystyle{plain}
\bibliography{references}
\end{document}